\documentclass[reqno]{amsart}

\makeatletter
\let\origsection=\section \def\section{\@ifstar{\origsection*}{\mysection}} 
\def\mysection{\@startsection{section}{1}\z@{.7\linespacing\@plus\linespacing}{.5\linespacing}{\normalfont\scshape\centering\S}}
\makeatother  

\usepackage{amssymb}
\usepackage{amsmath}
\usepackage{amsthm}
\usepackage{letterswitharrows}
\usepackage{mathrsfs}
\usepackage{graphicx}
\usepackage{import}
\usepackage{latexsym}
\usepackage{stmaryrd}
\usepackage{enumitem}
\usepackage{moreenum}
\usepackage[bookmarks,hyperfootnotes=false, psdextra=true]{hyperref}
\usepackage{multirow}
\usepackage{xcolor}
\usepackage{caption}
\colorlet{darkishRed}{red!60!black}
\colorlet{darkishBlue}{blue!60!black}
\colorlet{darkishGreen}{green!50!black}
\colorlet{lightishGreen}{green!70!black}
\hypersetup{
    draft = false,
    bookmarksopen=true,
    colorlinks,
    linkcolor={darkishGreen},
    citecolor={lightishGreen},
    urlcolor={darkishBlue}
}
\usepackage[nameinlink, capitalise, noabbrev]{cleveref}
\usepackage{thmtools}
\usepackage{nameref}
\crefformat{enumi}{#2#1#3}
\crefformat{equation}{#2(#1)#3}
\crefrangeformat{section}{Sections~#3#1#4--#5#2#6}
\crefname{mainresult}{Theorem}{Theorems}
\usepackage[msc-links,abbrev]{amsrefs} 
\usepackage{doi}

\renewcommand{\PrintDOI}[1]{\doi{#1}}

\let\setminus=\smallsetminus
\usepackage{tikz}
\usepackage{tikz-cd}
\usetikzlibrary{calc,through,intersections,arrows, trees, positioning, decorations.pathmorphing, cd}
\usepackage{comment}
\usepackage{mathtools}
\usepackage{nccmath}
\usepackage{pifont}
\usepackage[utf8]{inputenc}
\usepackage[T1]{fontenc}
\usepackage{lmodern}
\usepackage[babel]{microtype}
\usepackage[english]{babel}
\usepackage{relsize}

\newcommand{\COMMENT}[1]{{}}

\usepackage{geometry}
\let\setminus=\smallsetminus
\renewcommand{\leq}{\leqslant}
\renewcommand{\geq}{\geqslant}

\renewcommand{\le}{\leq}
\newcommand{\cleq}{\preccurlyeq}

\let\rho=\varrho
\let\phi=\varphi

\newcommand{ \R } { \mathbb{R} }

\newcommand{ \N } { \mathbb{N} }

\newcommand{\defn}[1]{{\color{darkishRed}{\emph{#1}}}}
\newcommand{\defnwe}[1]{{\color{darkishRed}{#1}}}

\newcommand{\red}[1]{{\color{red}{#1}}}
\newcommand{\blue}[1]{{\color{blue}{#1}}}

\makeatletter

\def\calCommandfactory#1{%
   \expandafter\def\csname c#1\endcsname{\mathcal{#1}}}
\def\frakCommandfactory#1{%
   \expandafter\def\csname frak#1\endcsname{\mathfrak{#1}}}

\newcounter{ctr}
\loop
  \stepcounter{ctr}
  \edef\X{\@Alph\c@ctr}
  \expandafter\calCommandfactory\X
  \expandafter\frakCommandfactory\X
  \edef\Y{\@alph\c@ctr}
  \expandafter\frakCommandfactory\Y
\ifnum\thectr<26
\repeat

\setenumerate{label={\normalfont (\roman*)}}

\newtheorem{theorem}{Theorem}[section] 
\newtheorem{proposition}[theorem]{Proposition}
\newtheorem{corollary}[theorem]{Corollary}
\newtheorem{lemma}[theorem]{Lemma}

\newtheorem{observation}[theorem]{Observation}

\newtheorem{mainresult}{Theorem}

\newtheorem{mainconjecture}[mainresult]{Conjecture}

\newtheorem{innercustomthm}{}

\newenvironment{customthm}[1]
  {\renewcommand\theinnercustomthm{#1}\innercustomthm}
  {\endinnercustomthm}

\theoremstyle{definition}
\newtheorem{example}[theorem]{Example}

\newtheorem{construction}[theorem]{Construction}

\theoremstyle{remark}

\usepackage{etoolbox}
\newbool{pdfBool}
\booltrue{pdfBool} 

\newcommand{\pdfOrNot}[2]{\ifbool{pdfBool}{{#1}}{{#2}}}

\usepackage{svg}

\newbool{arXiv}
\booltrue{arXiv} 

\newcommand{\arXivOrNot}[2]{\ifbool{arXiv}{{#1}}{{#2}}}

\newbool{TrackBool}
\booltrue{TrackBool} 

\newcommand{\change}[2]{\ifbool{TrackBool}{\blue{#1} \red{#2}}{{#1}}}

\DeclareMathOperator{\orw}{oradw}
\DeclareMathOperator{\rs}{rads}
\DeclareMathOperator{\irs}{irads}

\newcommand{\dist}{d}
\DeclareMathOperator{\rad}{rad}

\DeclareMathOperator{\Forb}{Forb}
\newcommand{\Forbminor}{\Forb_\cleq}

\newcommand{\td}{tree-decom\-pos\-ition}
\newcommand{\gd}{graph-decom\-pos\-ition}
\newcommand{\Hd}[1]{#1-decom\-pos\-ition}

\newcounter{mylabelcounter}

\makeatletter
\newcommand{\labelText}[2]{%
#1\refstepcounter{mylabelcounter}%
\immediate\write\@auxout{%
  \string\newlabel{#2}{{1}{\thepage}{{\unexpanded{#1}}}{mylabelcounter.\number\value{mylabelcounter}}{}}%
}%
}

\usepackage[skip=6pt]{subcaption} 
\begin{document}

\setlength{\fboxsep}{0pt}
\setlength{\fboxrule}{.1pt}

\title[A characterisation of graphs quasi-isometric to $K_4$-minor-free graphs]{A characterisation of graphs quasi-isometric\\ to $K_4$-minor-free graphs}

\author[S.\ Albrechtsen \and R.\ W.\ Jacobs \and P.\ Knappe]{Sandra Albrechtsen, Raphael W. Jacobs, Paul Knappe}
\address{University of Hamburg, Department of Mathematics, Bundesstraße 55 (Geomatikum), 20146 Hamburg, Germany}
\email{\{sandra.albrechtsen, raphael.jacobs, paul.knappe\}@uni-hamburg.de}
\author[P.\ Wollan]{Paul Wollan}
\address{University of Rome, “La Sapienza”, Department of Computer Science, Via Salaria 113, 00198 Rome, Italy}
\email{wollan@di.uniroma1.it}


\begin{abstract}
    We prove that there is a function $f$ such that every graph with no $K$-fat $K_4$ minor is $f(K)$-quasi-isometric to a graph with no $K_4$ minor.
    This solves the~$K_4$-case of a general conjecture of Georgakopoulos and Papasoglu.
    Our proof technique also yields a new short proof of the respective~$K_4^-$-case, which was first established by Fujiwara and Papasoglu.
\end{abstract}

\maketitle

\section{Introduction}

\subsection{Quasi-isometry and fat minors}

Gromov's \cite{Gromov} coarse geometry viewpoint had a profound influence on the field of geometric group theory and has resonated in neighbouring areas of study. 
At its core, this perspective revolves around the concept of \emph{quasi-isometry}, a generalisation of bi-Lipschitz maps which allows for an additive error.
Roughly speaking, two metric spaces are quasi-isometric whenever their large scale geometry coincides (see \cref{sec:quasi-isometry} for the definition).

Following Gromov's idea into the realm of graphs, Georgakopoulos and Papasoglu \cite{georgakopoulos2023graph} recently presented results and questions regarding the interplay of geometry and graphs, which they hope evolves into a coherent \emph{Coarse Graph Theory} or \emph{Graph-Theoretic Geometry}.
At the heart of their paper, they proposed a conjecture \cite{georgakopoulos2023graph}*{Conjecture~1.1} which would offer, for some prescribed finite graph $X$, a characterisation of the graphs whose large scale geometry is that of a graph with no $X$ minor. This characterisation is in terms of `fat minors', a coarse variant of graph minors. Roughly speaking, a \emph{fat minor} is a minor with additional distance constraints, which in particular ensures its branch sets to be pairwise far apart; especially, $0$-fat minors are the usual minors (see \cref{sec:fatminors} for the definition).

\begin{mainconjecture}\label{conj:agelospanos} \cite{georgakopoulos2023graph}*{Conjecture~1.1}
    Let $X$ be a finite graph.
    Then there exists a function $f \colon \N \to \N \times \N$ such that every graph with no $K$-fat $X$ minor is $f(K)$-quasi-isometric to a graph with no $X$ minor.
\end{mainconjecture}

\noindent 
Note that the case~$K = 0$ of \cref{conj:agelospanos} holds trivially (with $f(0) := (1,0)$), since $0$-fat $X$ minors are precisely the $X$ minors.

We remark that it is easy to see that the (qualitative) converse of \cref{conj:agelospanos} holds for any (not even necessarily finite) graph $X$.
\medskip

Recently, \cref{conj:agelospanos} has been disproved by Davies, Hickingbotham, Illingworth and McCarty for general graphs~$X$ \cite{CounterexAgelosPanosConjecture}*{Theorem~1}.
However, \cref{conj:agelospanos} is known to be true for particular graphs~$X$: The case $X = K_3$, which characterises graphs quasi-isometric to a forest, follows from a result of Manning \cite{Manning} (see~\cite{georgakopoulos2023graph} or \cite{bergerseymourboundeddiamTD} for a graph-theoretic proof). The case $X = K_{2,3}$ was proved by Chepoi, Dragan, Newman, Rabinovich, and Vaxes \cite{FatK23Minor}, which characterises graphs quasi-isometric to an outerplanar graph. Further, Fujiwara and Papasoglu \cite{coarsecacti} proved the case $X = K_4^-$, a $K_4$ with an edge removed, which characterises graphs quasi-isometric to a cactus. 
The case $X = K_{1,m}$ was proved by Georgakopoulos and Papasoglu~\cite{georgakopoulos2023graph}. 
Additionally, together with Diestel, Elm and Fluck, we showed in \cite{radialpathwidth} a similar, but even stronger characterisations of graphs quasi-isometric to a disjoint union of paths or to a disjoint union of subdivided stars, respectively. 

\subsection{Our results}

The main contribution of this paper is the resolution of \cref{conj:agelospanos} for $X = K_4$:

\begin{mainresult}\label{main:FatK4}
    There exists a function $f \colon \N \to \N \times \N$ such that every graph with no $K$-fat $K_4$ minor is $f(K)$-quasi-isometric to a graph with no $K_4$ minor.
\end{mainresult}

\noindent It follows that every connected finite graph with no $K$-fat $K_4$ minor is either $(M,A)$-quasi-isometric to a $2$-connected series-parallel graph where $(M,A) := f(K)$ for the function $f$ from \cref{main:FatK4}, or it can be disconnected by removing a ball of radius $A$. 

In general, the construction of a quasi-isometry and the argument for its verification may be quite lengthy. For example, the proof of \cref{conj:agelospanos} for $X = K_4^-$ given by Fujiwara and Papasoglu \cite{coarsecacti} is more than 20 pages long.

There is an alternative framework in which quasi-isometries between graphs can be cast, which enables us to look at \cref{conj:agelospanos} from a different angle:
Graph-decompositions~\cite{GraphDec} are a natural extension of \td s which allow the bags $V_h$ of decompositions $(H, \cV)$ to be arranged along general decomposition graphs $H$ instead of just trees.
Recent applications of graph-decompositions include a local-global decomposition theorem~\cite{GraphDec} as well as the study of local separations~\cite{LocalSeps} and of locally chordal graphs~\cite{LocallyChordal}.
In \cite{radialpathwidth}*{\S~3.3} it was shown that two graphs are quasi-isometric if and only if each has a decomposition modelled on the other, subject to bounds on width parameters that correspond to the constants of the quasi-isometry (see \cref{subsec:GD,subsec:InterplayQIGD} for details).
To prove \cref{main:FatK4}, we construct such a decomposition of~$G$ modelled on a graph $H$ with no $K_4$ minor instead of defining a quasi-isometry from $G$ to $H$.
This in particular yields a more graph-theoretic flavour of the arguments used in our proof of \cref{main:FatK4}.

Our techniques additionally yield a new proof of \cref{conj:agelospanos} for $X = K_4^-$, which is significantly shorter than the original one by Fujiwara and Papasoglu \cite{coarsecacti}:

\begin{mainresult} \label{main:Cactus}
    There exists a function $f \colon \N \to \N \times \N$ such that every graph with no $K$-fat $K_4^-$ minor is $f(K)$-quasi-isometric to a graph with no $K_4^-$ minor.
\end{mainresult}

\noindent Recall that the connected graphs with no $K_4^-$ minor are precisely the cacti, i.e.\ those graphs in which any two cycles meet in at most one vertex. In particular, every connected graph with no $K$-fat $K_4^-$ minor is $f(K)$-quasi-isometric to a cactus where $f$ is the function from \cref{main:Cactus}.
%
We also remark that our proof yields a function $f$ which is quite different from the one given by Fujiwara and Papasoglu.\footnote{
For a given graph $G$ with no $K$-fat $K_4^-$ minor, we construct a $(84K+2, 84K+2)$-quasi-isometry from some cactus $H$ to~$G$ (which then by \cref{lem:inversequasiisom} yields that $G$ is quasi-isometric to $H$).
In contrast to that Fujiwara and Papasoglu construct a $(2, 6000 \cdot 10^{100} K)$-quasi-isometry from some cactus $H'$ to $G$; their constants can be found in \cite{coarsecacti}*{Proofs of Lemma~1.9 and Theorem~2.1}. 
We remark that neither of us optimised their constants.
}
\medskip

Let us finally emphasise that we prove \cref{main:Cactus,main:FatK4} for arbitrary (finite or infinite) graphs\footnote{This made no difference in the proofs except for \cref{lem:ComponentAttachingToTwoBalls}.}. Thus, \cref{main:FatK4,main:Cactus} transfer verbatim to length spaces, a generalisation of geodesic metric spaces \cite{georgakopoulos2023graph}*{Observation~2.1}.

\subsection{Applications}

Bonamy, Davies, Esperet and Wesolek brought the following consequences of \cref{main:FatK4} to our attention \cite{DavisPersonalCommunication}.
We thank them for allowing us to include those applications in this paper.

Gromov \cite{Gromov} introduced \defn{asymptotic dimension} of metric spaces in the context of geometric group theory; see \cite{AsDimSurvey} for a survey on asymptotic dimension and its group-theoretic applications.\footnote{For the various equivalent definitions of `asymptotic dimension', we refer the reader to \cite{asymptoticdimminorclosed}*{\S~1.1}.
For example, the \defn{asymptotic dimension} of a graph class $\cG$ is the least $n \in \N$ such that there exists a function $f$ such that for all $G \in \cG$ and for all $r \in \N$, $V(G)$ can be coloured with $n+1$ colours so that for all $v, u \in V(G)$, if they are connected by a monochromatic path in $G^r$, then the distance between $v,u$ in G is $\leq f(r)$ \cite{asymptoticdimminorclosed}*{Proposition~1.17}.}
Bonamy, Bousquet, Esperet, Groenland, Liu, Pirot and Scott proved that for any finite graph $X$, the class of graphs with no $X$ minor has asymptotic dimension at most $2$ \cite{asymptoticdimminorclosed}*{Theorem~2}, and they asked whether this can be generalised to classes of graphs excluding a graph $X$ as a $K$-fat minor for some $K \in \N$ \cite{asymptoticdimminorclosed}*{Question~5}. 
It now follows from \cref{main:FatK4} that this is true for $X = K_4$.

\begin{theorem}\label{theorem:NoFatK4yieldsAsDim1}
    For every $K \in \N$, the class of graphs with no $K$-fat $K_4$ minor has asymptotic dimension at most $1$.
\end{theorem}

\begin{proof}
    Let $\cH$ be the class of all graphs with no $K_4$ minor. Since $K_4$ is planar, $\cH$ has asymptotic dimension at most $1$ by \cite{asymptoticdimminorclosed}*{Theorem~1.2}.
    Let $K \in \N$.
    By \cref{main:FatK4}, every graph with no $K$-fat $K_4$ minor is $f(K)$-quasi-isometric to a graph in $\cH$ where $f$ is the function from \cref{main:FatK4}.
    Since asymptotic dimension is preserved under quasi-isometry  \cite{AsDimSurvey}*{Proposition~22}, the class of graphs with no $K$-fat $K_4$ minor thus has asymptotic dimension at most $1$, too.
\end{proof}

A class of graphs is \defn{hereditary} if it is closed under taking induced subgraphs.
Such a class of graphs is \defn{$\chi$-bounded} if there exists a function $f \colon \N \to \N$ such that $\chi(G) \leq f(\omega(G))$ for every graph $G$ in the class, where $\chi$ and $\omega$ denote the \emph{chromatic number} and the \emph{clique number}, respectively.

Burling \cite{burling} constructed a sequence of  graphs with increasing chromatic number that contain no $K_3$ as a subgraph; for a modern study of these graphs, see \cite{burlingrevisitedI}.
In the following, we will refer to the induced subgraphs of graphs in the Burling sequence as \defn{Burling graphs}. 
From \cref{theorem:NoFatK4yieldsAsDim1} it now follows that the class of Burling graphs provides an example for a graph class which has bounded asymptotic dimension but is not $\chi$-bounded.

\begin{example}
    \emph{There exists a hereditary class of graphs which has asymptotic dimension at most $1$ and is not $\chi$-bounded.}
\end{example}

\begin{proof}
    The hereditary class of Burling graphs is not $\chi$-bounded, as Burling graphs do not contain the complete graph $K_3$, but have unbounded chromatic number.
    Hence, it remains to prove that the class of Burling graphs has asymptotic dimension at most $1$. For this, by \cref{theorem:NoFatK4yieldsAsDim1}, it suffices to show that Burling graphs do not contain $2$-fat $K_4$ minors.
    As Burling graphs do not contain $K_3$ as a subgraph, every $2$-fat $K_4$ minor in a Burling graph $G$ yields a subdivision of $K_4$ as an induced subgraph of $G$ such that its branch vertices have distance at least $2$.
    By \cite{burlingrevisitedII}*{Theorem~7.4}, every subdivision of a $K_4$ which is an induced subgraph of a Burling graph contains at least one non-subdivided edge.
    Hence, $G$ cannot contain a $2$-fat $K_4$ minor.
\end{proof}

\subsection{An open conjecture}

Even though \cref{conj:agelospanos} is disproved for general graphs $X$~\cite{CounterexAgelosPanosConjecture}, we would like to draw the reader's attention to the following special case\footnote{Note that one may replace the graph $X$ in \cref{conj:agelospanos} with any finite set of finite graphs.} of \cref{conj:agelospanos}, which is still open. It is a coarse variant of Kuratowski's theorem and was first conjectured by Georgakopoulos and Papasoglu \cite{georgakopoulos2023graph}:

\begin{mainconjecture}
    There exists a function $f: \N \rightarrow \N \times \N$ such that every graph with no $K$-fat $K_5$ and $K_{3,3}$ minor is $f(K)$-quasi-isometric to a planar graph.
\end{mainconjecture}

\subsection{How this paper is organised}

In \cref{sec:Prelims} we recall some important definitions and facts about the interplay between quasi-isometries and graph-decompositions of small radial width.
\cref{sec:reductionofconj:agelospanos} then contains some preparatory lemmas, which simplify the proofs of \cref{main:Cactus,main:FatK4}.
Finally, we prove \cref{main:Cactus} in \cref{sec:Cactus} and \cref{main:FatK4} in \cref{sec:SeriesParallel}.

\section{Preliminaries} \label{sec:Prelims}

We warn the reader that graphs in this paper may be infinite, unless specified otherwise.
Our notions mainly follow~\cite{DiestelBook}. In particular, if $P = p_0 \dots p_n$ is a path in a graph $G$, then we denote by \defn{$p_iPp_j$} for $i, j \in [n]$ the subpath $p_i p_{i+1} \dots p_j$ of $P$. Further, we denote by \defn{$\bar P$} the path $p_n \dots p_0$, and by \defn{$\mathring{P}$} the subpath $p_1Pp_{n-1}$ of $P$.
We recall that for two sets $U_1,U_2$ of vertices of $G$, an \defn{$U_1$--$U_2$ path} meets $U_1$ precisely in its first vertex and $U_2$ precisely in its last vertex.
For a subgraph $C$ of $G$, a \defn{$C$-path} is a non-trivial path which meets $C$ precisely in its endvertices.
A path \defn{through} a set $U$ of vertices of~$G$ is a path in $G$ whose internal vertices are contained in $U$.

We denote by \defn{$\cC(G)$} the \defn{set of components} of the graph $G$.
Given sets $U' \subseteq U$ of vertices of $G$, a component $C$ of $G-U$ \defn{attaches} to $U'$ if $C$ has a neighbour in $U'$.
The \defn{boundary $\partial_G Y$} of a subgraph $Y$ of $G$ is the set $N_G(V(G-Y))$ of vertices of $Y$ that send in~$G$ an edge outside of $Y$. For example, the boundary $\partial_G C$ of a component of $G-U$ is $N_G(U) \cap V(C)$.

\subsection{Distance, radius and balls}

Let $G$ be a graph.
We write~\defn{$\dist_G(v, u)$} for the distance of the two vertices~$v$ and~$u$ in~$G$. 
For two sets~$U$ and~$U'$ of vertices of~$G$, we write~$\dist_G(U, U')$ for the minimum distance of two elements of~$U$ and~$U'$, respectively.
If one of~$U$ or~$U'$ is just a singleton, then we may omit the braces, writing $\dist_G(v, U') := \dist_G(\{v\}, U')$ for $v \in V(G)$.

Given a set~$U$ of vertices of~$G$, the \defn{ball (in~$G$) around~$U$ of radius $r \in \N$}, denoted as~\defn{$B_G(U, r)$}, is the set of all vertices in~$G$ of distance at most~$r$ from~$U$ in~$G$.
If~$U = \{v\}$ for some~$v \in V(G)$, then we omit the braces, writing~$B_G(v, r)$ for the ball (in $G$) around~$v$ of radius~$r$.
Additionally, we abbreviate the induced subgraph on $B_G(U,r)$ of $G$ with $\defnwe{G[U,r]} := G[B_G(U,r)]$.

Further, the \defn{radius}~\defn{$\rad(G)$} of~$G$ is the smallest number~$k \in \N \cup \{\infty\}$ such that there exists some vertex~$w \in V(G)$ with~$\dist_G(w, v) \le k$ for every vertex~$v$ of $G$. If $G$ is empty, then we define its radius to be~$0$.
We remark that if $G$ is disconnected but not the empty graph, then its radius is $\infty$.
Note that~$G$ has radius at most~$k$ if and only if there is some vertex~$v$ of $G$ with~$V(G) = B_G(v, k)$.
Additionally, if $U \subseteq V(G)$, then the \defn{radius of $U$ in $G$}, denoted as $\defnwe{\rad_G(U)}$ is the smallest number $k \in \N$ such that there exists some vertex $v$ of $G$ with $U \subseteq B_G(v, k)$ or $\infty$ if such a $k \in \N$ does not exist. 

If $Y$ is a subgraph of $G$, then we abbreviate $\dist_G(U,V(Y))$, $\rad_G(V(Y))$, $B_G(V(Y),r)$, and $G[V(Y),r]$ as $\dist_G(U,Y)$, $\rad_G(Y)$, $B_G(Y,r)$, and $G[Y,r]$, respectively.

\subsection{Fat minors}\label{sec:fatminors}

In this paper we deviate from the definition of minor in \cite{DiestelBook}, and consider the following equivalent definition:
Let $G, X$ be graphs.
A \defn{model} $(\cV,\cE)$ of $X$ in $G$ is a collection $\cV$ of disjoint sets $V_x \subseteq V(G)$ for vertices $x$ of $X$ such that each $G[V_x]$ is connected, and a collection $\cE$ of internally disjoint $V_{x_0}$--$V_{x_1}$ paths $E_{e}$ for edges $e=x_0x_1$ of $X$ which are disjoint from every $V_x$ with $x \neq x_0, x_1$.\footnote{
By enlarging the branch sets $V_x$ along the `adjacent' branch paths $E_{xy}$, we obtain that this notion of model is equivalent to the notion of model from \cite{DiestelBook}.}
The $V_x$ are its \defn{branch sets} and the $E_e$ are its \defn{branch paths}.
A model $(\cV, \cE)$ of $X$ in $G$ is \defn{$K$-fat} for $K \in \N$ if $\dist_G(Y,Z) \geq K$ for every two distinct $Y,Z \in \cV \cup \cE$ unless $Y = E_e$ and $Z = V_x$ for some vertex $x \in V(X)$ incident to $e \in E(X)$, or vice versa.
Then $X$ is a \defn{($K$-fat) minor} of $G$ if $G$ contains a ($K$-fat) model of $X$.
We remark that the $0$-fat minors of $G$ are precisely its minors.
By \defnwe{$\Forbminor(X)$}, we denote the class of all graphs with no $X$ minor.

\subsection{Quasi-isometries} \label{sec:quasi-isometry}

For~$M \in \R_{\geq 1}$ and~$A \in \R_{\geq 0}$, an \defn{$(M,A)$-quasi-isometry} from a graph~$H$ to a graph~$G$ is a map~$\phi\colon V(H) \to V(G)$ such that

\begin{enumerate}[label=(Q\arabic*)]
    \item \label{quasiisom:1} $M^{-1} \cdot \dist_H(h, h') - A \leq \dist_G(\phi(h),\phi(h')) \leq M \cdot \dist_H(h,h') + A$ for every~$h, h' \in V(H)$, and
    \item \label{quasiisom:2} for every vertex $v$ of $G$, there exists a node $h$ of $H$ with $\dist_G(v,\phi(h)) \leq A$.
\end{enumerate}
We say that $H$ is \defn{$(M,A)$-quasi-isometric} to $G$ if there exists a $(M,A)$-quasi-isometry from $H$ to $G$.
The following is a well-known fact. 

\begin{lemma}\label{lem:inversequasiisom}
    If a graph $H$ is $(M,A)$-quasi-isometric to a graph $G$, then $G$ is $(M,3AM)$- quasi-isometric to $H$. \qed
\end{lemma}

\subsection{Graph-decompositions}\label{subsec:GD}

Let~$G$ and~$H$ be graphs and let~$\cV=(V_h)_{h\in H}$ be a family of sets~$V_h$ of vertices of~$G$.
We call $(H, \cV)$ an \defn{\Hd{$H$}} of~$G$, a \defn{decomposition} of $G$ \defn{modelled on $H$}, or just a \defn{\gd} \cite{GraphDec}, if
\begin{enumerate}[label=(H\arabic*)]
    \item \label{GraphDecomp:H1} $\bigcup_{h\in H} G[V_h] = G$, and 
    \item \label{GraphDecomp:H2} for every vertex~$v$ of $G$, the graph~$H_v := H[\{h\in V(H)\mid v\in V_h\}]$ is connected. 
\end{enumerate}
The sets~$V_h$ are called the \defn{bags} of this \gd, their induced subgraphs~$G[V_h]$ are its \defn{parts}, and the graph~$H$ is its~\defn{decomposition graph}.
Whenever a \gd\ is introduced as~$(H, \cV)$, we tacitly assume~$\cV = (V_h)_{h \in H}$.
A graph-decomposition $(H,\cV)$ is \defn{honest}, if all its bags $V_h$ are non-empty and for every edge $h_0h_1$ of $H$ the bags $V_{h_0}$ and $V_{h_1}$ intersect non-emptily.
By \cite{radialpathwidth}*{Lemma~3.2}, every \gd\ not only satisfies \cref{GraphDecomp:H2} but also
\begin{enumerate}[label=(H\arabic*')]
    \setcounter{enumi}{1}
    \item \label{GraphDecomp:H2'} for any connected subgraph~$Y$ of~$G$, the graph~$H_Y := H[\{h \in V(H) \mid V(Y) \cap V_h \neq \emptyset\}]$ is connected. 
\end{enumerate}

A graph-decomposition of an induced subgraph $Y$ of $G$ is a \defn{partial graph-decomposition} of $G$ with \defn{support}~$Y$. 
Note that every partial graph-decomposition of $G$ with support $Y = G$ is a \gd\ of $G$ and vice versa.

Let $G$ be a graph and let $\cH = (H,\cV)$ be a partial graph-decomposition of $G$ with support $Y$. The \defn{outer-radial width} of $(H, \cV)$ is
\begin{equation*}
    \orw_G(\cH) := \sup_{h \in H} \rad_G(V_h),
\end{equation*}
and the \defn{(inner-)radial spread} of $(H, \cV)$ is
\begin{equation*}
    \irs(\cH) := \sup_{v \in Y} \rad_{H_v}(H_v).
\end{equation*}
Additionally, we set $\irs(\cH,v) := \rad_{H_v}(H_v)$ for a vertex $v$ of $Y$.
To shorten notation, we say that a partial graph-decomposition $\cH$ of a graph $G$ is \defn{$(R_0,R_1)$-radial} in $G$ for $R_0, R_1 \in \N$, if $ \orw_G(\cH) \leq R_0$ and $\irs(\cH)\leq R_1$.

We remark that the inner-radial spread of $(H,\cV)$ is at least the \defn{outer-radial spread} $\rs(\cH) = \sup_{v \in Y} \rad_H(H_v)$ of $(H,\cV)$, which we have used in \cite{radialpathwidth}*{\S~3.3}.
The reason why we prefer to work with inner-radial spread in this paper is that the inner-radial spread of a \gd\ $(H,\cV)$ does not change if we delete edges $h_0h_1$ of $H$ whose corresponding intersection $V_{h_0} \cap V_{h_1}$ of bags is empty.

The following observation is immediate from the definitions.

\begin{observation} \label{obs:Connected}
    Let $G$ be a graph and, for every component $C$ of $G$, let $\cH^C = (H^C, \cV^C)$ be a \gd\ of $C$. Then $\cH = (H, \cV)$ is a \gd\ of $G$ where $H$ is the disjoint union of  $H^C$ and $V_h := V^C_h$ for all $h \in V(H)$ and the unique $C \in \cC(G)$ such that $h \in V(H^C)$.

    In particular, $\orw_G(\cH) = \sup_{C \in \cC(G)}\orw_{C}(\cH^C)$ and $\irs(\cH) = \sup_{C \in \cC(G)}\irs(\cH^C)$. Moreover, if, for some connected graph $X$, we have $H^C \in \Forbminor(X)$ for all $C \in \cC(G)$, then also $H \in \Forbminor(X)$. \qed
\end{observation}

\subsection{Interplay between quasi-isometries and graph-decompositions}\label{subsec:InterplayQIGD}

By \cite{radialpathwidth}*{\S~3.3}, a graph $H$ is quasi-isometric to a graph $G$ if and only if $G$ admits an honest decomposition modelled on~$H$ of bounded outer-radial width and bounded radial spread, where the constants of the quasi-isometry only depend on the outer-radial width and the radial spread of the \gd\ and vice versa.

\begin{lemma}{\cite{radialpathwidth}*{Lemma 3.9}} \label{lem:GraphDecToQuasiIso}
    If a graph $G$ admits an honest $(R_0, R_1)$-radial decomposition modelled on a graph $H$, then there is a $(\max\{2R_0, 2R_1\}, \max\{2R_0, 2R_1\})$-quasi-isometry from $H$ to $G$.
\end{lemma}

\begin{lemma}{\cite{radialpathwidth}*{Lemma 3.10}}\label{lem:QuasiIsoToGraphDec}
    If a graph $H$ is $(M,A)$-quasi-isometric to some graph $G$ for $M \in \N_{\geq 1}$ and $A \in \N$, then $G$ admits an honest $(A (M^2+2) + M \lceil M(A+1)/2 \rceil, 6MA + M + 1)$-radial $H$-decomposition. 
\end{lemma}

By \cref{lem:GraphDecToQuasiIso,lem:QuasiIsoToGraphDec}, \cref{conj:agelospanos} can thus be reformulated in terms of \gd s in the following qualitatively equivalent way.

\begin{customthm}{\cref*{conj:agelospanos}$^\prime$} \label{conj:agelospanosGD}
    Let $X$ be a finite graph.
    Then there exist a function $f \colon \N \to \N \times \N$ such that every graph~$G$ with no $K$-fat $X$ minor admits an honest $f(K)$-radial decomposition $(H,\cV)$ modelled on a graph~$H$ with no~$X$ minor.
\end{customthm}

\noindent We will later refer to the coordinates of $f$ as $f_0$ and $f_1$ in the sense of $f(K) = (f_0(K), f_1(K))$ and vice versa.

\section{Preparatory work: Sequences of graph-decompositions}\label{sec:reductionofconj:agelospanos}

As we have mentioned already in the introduction, we will prove \cref{main:Cactus,main:FatK4} by showing the equivalent statements about \gd s, that is, we will prove \cref{conj:agelospanosGD} for $X = K_4^-$ and $X = K_4$. For this we construct for a given graph $G$ the desired \gd\ of $G$ recursively by `extending' partial \gd s of $G$ step-by-step to one of $G$. More precisely, given a (suitable) partial \gd\ $(H^d, \cV^d)$ of $G$ with support $Y^d$ we will extend it to a $(R_0, R_1)$-radial partial \gd\ $(H^{d+1}, \cV^{d+1})$ of $G$ with support $Y^{d+1} \supseteq Y^d$. Here, we will ensure that $N_G(Y^d) \subseteq V(Y^{d+1})$, so that the desired \gd\ may arise as the `limit' of the $(H^d, \cV^d)$.

In this section we will formalise this procedure and establish a theorem (\cref{thm:CompExtensionLem}) that will yield the \gd\ at once if we can do one simplified `extension step'; see \textup{\nameref{tag:star}} below.
\medskip

The \defn{restriction} $(H, \cV)$ of a \gd\ $(H, \cV')$ of $G$ to a subgraph $Y$ of $G$ is given by $V_h := V'_h \cap V(Y)$.
Conversely, a partial \gd\ $(H, \cV)$ of $G$ with support $Y$ \defn{extends} to a partial \gd\  $(H',\cV')$ of $G$ with support $Y'$ if $H \subseteq H'$, $Y \subseteq Y'$, and $V_h \subseteq V'_h$ for every node $h$ of $H$. 
Note that if $(H, \cV)$ and $(H, \cV')$ are partial \gd s of $G$ with supports $Y$ and $Y'$, respectively, such that $Y \subseteq Y'$, and with the same decomposition graph $H$, then $(H, \cV)$ extends to $(H, \cV')$ if and only if $(H, \cV)$ is the restriction of $(H, \cV')$ to $Y$.

Let $(\cH^d)_{d \in \N}$ be a sequence of partial \gd s $\cH^d = (H^d, \cV^d)$ of $G$ with supports $Y^d$ such that $\cH^{d}$ extends to $\cH^{d+1}$ for every $d \in \N$.
Then we set $Y := \bigcup_{n \in \N} Y^d$, $H := \bigcup_{d \in \N} H^d$ and $V_h := \bigcup_{d \in \N: h \in H^d} V_h^d$ for every node $h$ of $H$.
We call the pair $\cH = (H, \cV)$ the \defn{limit} of $(\cH^d)_{d \in \N}$.

The following observation is immediate from the definitions:

\begin{observation}\label{lem:LimitGD}
    Let $(\cH^d)_{d \in \N}$ be a sequence of partial \gd s of a graph $G$ such that $\cH^{d}$ extends to $\cH^{d+1}$ for every $d \in \N$.
    Then the limit $\cH$ is a partial \gd\  of $G$ with support $Y = \bigcup_{d \in \N} Y^d$.
    Moreover,
    \begin{enumerate}
        \item if, for a finite graph $X$, $H^d \in \Forbminor(X)$ for every $d \in \N$, then $H \in \Forbminor(X)$,
        \item if $\cH^d$ is honest for every $d \in \N$, then $\cH$ is honest,
        \item $\rad_G(V_h) \leq \sup_{d : h \in H^d} \rad_G(V^d_h)$ for every node $h$ of $H$, and
        \item $\irs(\cH,v) \leq \sup_{d : v \in Y^d} \irs(\cH^d,v)$ for every vertex $v$ of $Y$. \qed
    \end{enumerate}
\end{observation}

First, we describe how to extend a partial \gd\ $(H^d, \cV^d)$ with support $Y_d$ if it is possible to extend it in the direction of each remaining component of $G- Y_d$ individually,
that is if we are given for each $C \in \cC(G-Y^d)$ a partial \gd\ $\cH^C = (H^C, \cV^C)$ of $G$ with support $Y^C$ such that $N_G(C) \subseteq V(Y^C) \subseteq B_G(C,1)$.

Let $\cH = (H, \cV)$ be a partial \gd\ of $G$ with support $Y$.
For $C \in \cC(G-Y)$, a partial \gd\  $\cH^C = (H^C, \cV^C)$ of $G$ with support $Y^C \subseteq G[C,1]$ is \defn{feasible for $\cH$ and $C$}, if $\partial_G C \subseteq V(Y^C)$ and there exist nodes $h^C$ of $H^C$ and $g^C$ of $H$ with $N_G(C) = V^C_{h^C}$ and $N_G(C) \subseteq V_{g^C}$.
A family $\{\cH^C \mid C \in \cC(G-Y)\}$ of partial \gd s of $G$ is \defn{feasible for $\cH$} if each $\cH^C$ is feasible for $\cH$ and $C$.

\begin{construction}\label{constr:GraphDecompViaComps}
    Assume that $\cH = (H, \cV)$ is a
    partial \gd\ of a graph $G$ with support~$Y$. 
    Let $\{\cH^C \mid C \in \cC(G-Y)\}$ be a family of partial \gd s  $\cH^C = (H^C, \cV^C)$ of $G$ with support $Y^C$ which is feasible for $\cH$.

    We define the pair $\cH' := (H', \cV')$ as follows.
    Let $H'$ be the graph obtained from the disjoint union of~$H$ and the $H^C$ by identifying $h^C \in V(H^C)$ with $g^C \in V(H)$ for every $C \in \cC(G-Y)$.
    For each node $h$ of $H'$, we set $V_h' := V_h$ if $h \in V(H)$ and $V_h' := V_h^C$ if $h \in V(H^C) \setminus \{h^C\}$ for some $C \in \cC(G-Y)$.
    Set $Y' := Y \cup \bigcup_{C \in \cC(G-Y)} Y^C$.
\end{construction}

Note that the graph $H'$ from \cref{constr:GraphDecompViaComps} is obtained from $H$ and the $H^C$ by $1$-sums. Here, the \defn{$1$-sum} of two graphs $G_0$ and $G_1$ with respect to vertices $v_0$ of $G_0$ and $v_1$ of $G_1$ is obtained from the disjoint union of $G_0$ and $G_1$ by identifying $v_0$ with $v_1$. We remark that $\Forbminor(X)$ is closed under $1$-sums if and only if $X$ is $2$-connected. Hence, if $H$ and the $H^C$ have no $X$-minor for $X = K_4$ or $X = K_4^-$, then also the arising graph $H'$ has no $X$ minor\footnote{One could, more generally, obtain a partial \gd\ $(H', \cV')$ by identifying more than one pair of vertices of $H$ and an $H^C$ if the respective bags of $H$ and $H^C$ are suitable for this. However, since we need to ensure for the proof of \cref{main:Cactus,main:FatK4} that the arising graph $H'$ has no $K_4^-$- or $K_4$ minor, we will not do this here.}.

The following observation about $\cH'$ from \cref{constr:GraphDecompViaComps} is immediate from its construction.

\begin{observation} \label{obs:AboutConstr}
    Let $G, \cH$ and the $\cH^C$ be as in \cref{constr:GraphDecompViaComps}. Then $\cH'$ constructed as in \cref{constr:GraphDecompViaComps} from $\cH$ and the $\cH^C$ is a partial \gd\   of~$G$ with support $Y'$, which extends $\cH$.
    Moreover,
    \begin{enumerate}
        \item if $H, H^C \in \Forbminor(X)$ for all $C \in \cC(G-Y)$ and some $2$-connected graph $X$, then $H' \in \Forbminor(X)$,
        \item if $\cH$ and $\cH^C$ are honest for every $C \in \cC(G-Y)$, then $\cH'$ is honest,
        \item $\rad_G(V'_h) = \rad_G(V_h)$, if $h \in V(H)$, 
        \item $\rad_G(V'_h) = \rad_G(V^C_h)$, if $h \in V(H^C) \setminus \{h^C\}$ for some $C \in \cC(G-Y)$,
        \item $\irs(\cH',v) = \irs(\cH,v)$ for $v \in V(Y) \setminus \partial_G Y$,
        \item $\irs(\cH',v) = \irs(\cH^C, v)$ for $v \in V(Y^C) \setminus \partial_G Y = V(Y^C \cap C)$ for $C \in \cC(G-Y)$, and 
        \item $\irs(\cH', v) \leq \irs(\cH,v) + 2 \sup_{C \in \cC(G-Y)} \irs(\cH^C,v)$ for $v \in \partial_G Y$. \qed
    \end{enumerate}
\end{observation}

A partial \gd\ $\cH$ of $G$ with support $Y$ is \defn{$R$-component-feasible} for $R \in \N$ if for every component $C$ of $G-Y$ there is $h \in V(H)$ with $N_G(C) \subseteq V_h$ and $\rad_G(V_h) \leq R$.
A subgraph $Y$ of~$G$ is \defn{$R$-ball-componental} if for every component $C$ of $G-Y$ there exists a vertex $w$ of $G$ such that $C$ is a component of $G - B_G(w,R)$.

The following lemma shows that we may assume without loss of generality that every $R$-component-feasible partial \gd\ has an $R$-ball-componental support.

\begin{lemma}\label{lem:CompFeasibleToBallComp}
    For every $R \in \N$, every honest $R$-component-feasible partial \gd\ $\cH$ of a graph $G$ with support $Y$ extends to an honest $R$-component-feasible partial \gd\  $\cH'$ of~$G$ modelled on the same decomposition graph with $R$-ball-componental support $Y'$ and $\orw(\cH') \leq \max\{\orw(\cH), R\}$, $\irs(\cH',v) = \irs(\cH,v)$ for $v \in V(Y)$, and $\irs( \cH',v) = 0$ for $v \in V(Y')\setminus V(Y)$.
\end{lemma}

\begin{proof}
    For each node $h$ of $H$ with $\rad_G(V_h) \leq R$, we fix a vertex $v_h$ of $G$ such that $V_h \subseteq B_G(v_h, R)$.
    Since $(H, \cV)$ is $R$-component-feasible, we may further fix for each $C \in \cC(G-Y)$, a node $h_C$ of $H$ such that $N_G(C) \subseteq V_{h_C}$ and $\rad_G(V_{h_C}) \leq R$.
    Set $U := V(Y) \cup \bigcup_{C \in \cC(G-Y)} (B_G(v_{h_C},R) \cap V(C))$. It is immediate from the construction that every component $C'$ of $G-U$ is also a component of $G-B_G(v_{h_C}, R)$ where $C$ is the unique component of $G-Y$ with $C' \subseteq C$. Hence, the induced subgraph $Y' := G[U]$ is $R$-ball-componental. 

    It remains to extend $(H, \cV)$ to the desired \gd\ $(H, \cV')$ of $Y'$.
    Set $V'_h := V_h \cup (B_G(v_h, R) \cap (\bigcup_{C \in \cC(G-Y) : h_C = h} V(C)))$ for every $h \in V(H)$.
    Note that, for every $h \in V(H)$ every vertex~$v$ of~$Y$ is contained in $V'_h$ if and only if it is contained in $V_h$, and every vertex $v \in U \setminus V(Y)$ is contained in exactly one $V'_h$.
    Thus, $(H, \cV')$ is indeed a \gd\ of $Y'$, and its outer-radial width and radial spread are as claimed.
    Moreover, by construction, $N_G(C') \subseteq V'_{h_C}$ where $C$ is the unique component of $G-Y$ with $C' \subseteq C$. Hence, $(H, \cV')$ is $R$-component-feasible.
\end{proof}

Recall that our overall goal is to construct, for a given graph $X$ and a graph $G$ without $K$-fat $X$ minor, a graph-decomposition of $G$ along a graph $H$ without $X$ minor. For this, we recursively construct $f(K)$-radial  partial decompositions $(H^d, \cV^d)$ with support $Y^d$, where each $(H^d, \cV^d)$ extends $(H^{d-1}, \cV^{d-1})$ `into the direction' of each component. Combining \cref{lem:LimitGD} and \cref{lem:CompFeasibleToBallComp}, it suffices to consider the case $Y^d = G[B]$ where $B$ is a ball of bounded radius and to `make progress' into the direction of one component of $G-Y^d$. The following result, \cref{thm:CompExtensionLem}, and property~\nameref{tag:star} formalise this.

We say that a graph $X$ has \defn{property \textup{\labelText{($*$)}{tag:star}}} if for every $K \in \N$ there are $R(K)$, $f'_0(K)$, $f'_1(K)$, $f''_1(K) \in \N$ with $R(K) \leq f'_0(K)$ such that, for every ball $B$ of radius $\leq R(K)$ in a graph $G$ with no $K$-fat $X$ minor and for every component $C$ of $G-B$, there exists an honest $R(K)$-component-feasible $(f'_0(K), f'_1(K))$-radial partial \gd\ $\cH^C$ of $G$ with support $Y^C$ modelled on a graph $H^C \in \Forbminor(X)$ such that $N_G(C) \cup \partial_G C \subseteq V(Y^C) \subseteq B_G(C,1)$, the neighbourhood $N_G(C)$ is a bag of $\cH^C$, and $\irs(\cH^C, v) \leq f''_1(K)$ for all $v \in N_G(C)$.

\begin{theorem}\label{thm:CompExtensionLem}    
    If a $2$-connected finite graph $X$ has property \textup{\nameref{tag:star}} (with respect to functions $R$, $f'_0$, $f'_1$, $f''_1$), then \cref{conj:agelospanosGD} holds for $X$ (with respect to $f := (f_0, f_1) := (f'_0, f'_1 + 2\cdot f''_1 + 1)$).
\end{theorem}

A partial \gd\ $\cH$ of $G$ with support $Y$ is \defn{$(R_0,R_1,R_1')$-radial} in $G$ if it is $(R_0,R_1)$-radial and $\irs(\cH,v) \leq R_1'$ for every $v \in \partial_G Y$.

\begin{proof}
    For $K \in \N$, let $R(K), f'_0(K), f'_1(K), f''_1(K)$ be given by property \nameref{tag:star} of $X$, and let $G$ be any graph with no $K$-fat $X$ minor. By \cref{obs:Connected}, we may assume that $G$ is connected. 
    
    We define a sequence of partial \gd s $(H^d, \cV^d)$ of $G$ with supports $Y^d$ as follows.
    Pick an arbitrary vertex $v$ of $G$ and initialise $Y^0 := G[v, R_0(K)]$. Let $(H^0, \cV^0)$ be the partial \gd\ of $G$ with support $Y^0$ whose decomposition graph $H^0$ is the graph on a single vertex $h$ and whose single bag is $V^0_h := V(Y^0)$. Clearly, $(H^0, \cV^0)$ is honest, $R(K)$-component-feasible and $(f_0(K), 0, 0)$-radial, and $H^0 \in \Forbminor(X)$.

    Now let $d \in \N$ be given, and assume that for $i \leq d$ we have already constructed a sequence of honest, $R(K)$-component-feasible, $(f_0(K), f_1(K), f'_1(K))$-radial partial \gd s $\cH^i := (H^i, \cV^i)$ of~$G$ with $R(K)$-ball-componental support $Y^i$ and with $H^i \in \Forbminor(X)$ such that each $\cH^{i+1}$ extends $\cH^i$. Let further $(H^C, \cV^C)$, for every $C \in \cC(G-Y^d)$, be given by property \nameref{tag:star} of $X$.
    
    Then applying \cref{constr:GraphDecompViaComps} to $\cH^d$ and the $(H^C, \cV^C)$ yields a partial \gd\ $(H', \cV')$ of $G$, which extends $\cH^d$, is honest and $(f_0(K), f_1(K), f'_1(K))$-radial by \cref{obs:AboutConstr}, and such that $H' \in \Forbminor(X)$, since~$X$ is $2$-connected.
    Moreover, $(H', \cV')$ is clearly still $R(K)$-component-feasible, so by \cref{lem:CompFeasibleToBallComp}, $(H', \cV')$ extends to an honest $R(K)$-component-feasible $(f_0(K), f_1(K), f'_1(K))$-radial partial \gd\ $(H^{d+1}, \cV^{d+1})$ of $G$ with $R(K)$-ball-componental support $Y^{d+1}$ and with $H^{d+1} = H' \in \Forbminor(X)$. 

    Now by \cref{lem:LimitGD}, the limit $(H, \cV)$ of the $(H^d, \cV^d)$ is an honest $(f_0(K), f_1(K))$-radial partial \gd\ of $G$ with support $Y := \bigcup_{d \in \N} Y^d$ and with $H \in \Forbminor(X)$. Since $Y^{d+1}$ contains $G[Y^d,1]$ by construction, and because $G$ is connected, we have $Y = G$, and hence $(H, \cV)$ is a \gd\ of~$G$. Thus, $(H, \cV)$ is as desired.
\end{proof}

\section{Proof of \texorpdfstring{\cref{main:Cactus}}{Theorem 3}: \texorpdfstring{\cref{conj:agelospanos}}{Conjecture 1} for \texorpdfstring{$X = K_4^-$}{X = K4-}} \label{sec:Cactus}

In this section we prove \cref{main:Cactus}, which we restate here in the terminology of \gd s.

\begin{customthm}{\cref*{main:Cactus}$^\prime$}
    \label{thm:RadialCactusWidth}
    Every graph $G$ with no $K$-fat $K_4^-$ minor for $K \in \N_{\geq 1}$ admits an honest $(42K+1, 28K+3)$-radial decomposition $(H,\cV)$ modelled on a graph $H$ with no $K_4^-$ minor.
\end{customthm}

\noindent Note that by \cref{lem:GraphDecToQuasiIso,lem:inversequasiisom}, \cref{thm:RadialCactusWidth} immediately implies \cref{main:Cactus}.
\medskip

To prove \cref{thm:RadialCactusWidth}, it suffices by \cref{thm:CompExtensionLem} to show that for every ball $B$ in $G$ of radius $\leq 42K+1$ and every component $C$ of $G-B$ there exists a $(42K+1, 28K+1)$-radial, $(42K+1)$-component-feasible partial decomposition $\cH^C$ of $G$ modelled on a cactus $H^d$ with support $Y^C \subseteq G[C,1]$ such that $\partial_G C \cup N_G(C) \subseteq V(Y^C)$ and $\irs(\cH^C, v) \leq 1$ for all $v \in N_G(C)$. 

For this, we first show that if $G$ has no $K$-fat $K_4^-$ minor, then no component $C$ of $G-B$ contains three vertices in its boundary $\partial_G C$ that are pairwise far apart in $G$ (see \cref{lem:FindingAK4MinusKFatMinor} and \cref{fig:FindingAK4MinusKFatMinor}). It follows that the boundary of every component $C$ of $G-B$ can be partitioned into two sets $N_1, N_2$ of small radius. If $N_1$ and $N_2$ are close in~$G$, then the boundary of $C$ has small radius, and we may obtain the desired partial \gd\ $\cH^C$ of $G$ by letting $H^C$ be a $K_2$ on vertices $h,h'$ and setting $V_h^C := N_G(C), V_{h'}^C := N_G(C) \cup \partial_G C$. Otherwise, the sets $N_1$ and $N_2$ are far apart in $G$. We then show that there is a partial path-decomposition~$\cH^C$ with support $Y^C \subseteq G[C,1]$ as desired `connecting' $N_1$ and $N_2$ such that its first and last bag are its only bags which intersect $B$ (see \cref{lem:PathDecomp} and \cref{fig:PathDecomp:1}).

\begin{lemma}\label{lem:FindingAK4MinusKFatMinor}
    Let $B$ be a ball in a graph $G$ around a vertex $w$, and let $C$ be a component of $G-B$.
    If $\partial_G C$ contains three vertices which are pairwise at least $7K$ apart in $G$ for some $K \in \N_{\geq 1}$, 
    then~$G$ has a $K$-fat $K_4^-$ minor.
\end{lemma}

\begin{proof}
    Let $r$ be the radius of the ball $B$ around $w$, and let $u_1, u_2, u_3$ be three vertices in $\partial_G C$ which are pairwise at least $7K$ apart in $G$.
    Then also $\dist_G(u_i, u_j) \leq 2r$ for $i \neq j$ since $B$ has radius $r$, and thus
    \[
        r \geq \dist_G(u_i,u_j)/2 \geq 7K/2 > 3K.
    \]
    Hence, $B_G(w, r-(3K-1))$ is non-empty, so we may pick for every $i \in [3]$ a shortest $v_i$--$B_G(w, r-(3K-1))$ path $Q^i = q^i_0 \dots q^i_{3K}$ in $G$.
    By defining the branch sets as
    \[
    V_1 := V(C),\; V_2 := B_G(w, r-(3K-1)),\; V_3 := V(q^1_{K}Q^1q^1_{2K}),\; V_4 := V(q^2_KQ^2q^2_{2K})
    \]
    and the branch paths as
    \[
    E_{13} := q_{2K}^1Q^1q_{3K}^1,\; E_{14} := q_{2K}^2Q^2q_{3K}^2,\; E_{23} := q^1_0Q^1q^1_{K},\; E_{24} := q_{0}^2Q^2q^2_{K},\; E_{12} := Q^3
    \]
    we obtain a $K$-fat $K_4^-$-model (see \cref{fig:FindingAK4MinusKFatMinor}). Indeed, the sets $V_i$ and the paths $E_{ij}$ obviously form a model of $K_4^-$, so it remains to show that this $K_4^-$-model is $K$-fat.

    \begin{figure}
        \centering
        \pdfOrNot{\def\svgwidth{0.6\columnwidth} 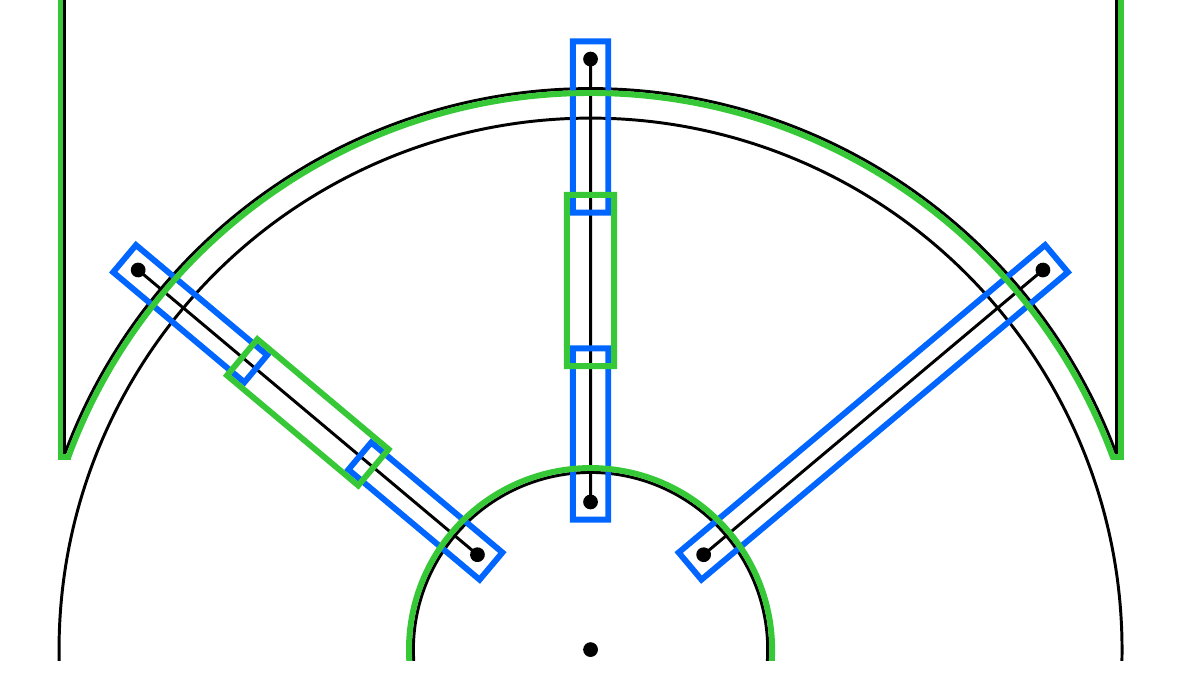}{\includesvg[width=0.6\columnwidth]{svg/k4minus.svg}}
        \caption{The $K$-fat $K_4^-$ minor in the proof of \cref{lem:FindingAK4MinusKFatMinor}.}
        \label{fig:FindingAK4MinusKFatMinor}
    \end{figure}

    Since $G[V_1] = C$ is a component of $G-B_G(w,r)$ and $V_2 = B_G(w, r-(3K-1))$, we have $\dist_G(V_1,V_2) \geq 3K \geq K$.
    In particular, as the $Q^i$ are shortest paths, we find $\dist_G(V_i, E_{(3-i)j}) \geq 2K \geq K$ as well as $\dist_G(E_{1j}, E_{2j}) \geq K$ for $i = 1,2$  and $j = 3,4$.
    By the choice of $V_3$, $V_4$ and the $E_{ij}$ as subsets of the $Q^i$, it remains to show $\dist_G(Q^i, Q^j) \geq K$ for $i \neq j$.
    Indeed, we find
    \[
        \dist_G(Q^i, Q^j) \geq \dist_G(u_i,u_j)-||Q_i||-||Q_j|| \geq 7K - 2 \cdot 3K = K,
    \]
    where the first inequality holds because any $Q^i$--$Q^j$ path extends via $Q^i$ and $Q^j$ to a $u_i$--$u_j$ path.
\end{proof}

For the next lemma, we need the following result, which is immediate from \cite{radialpathwidth}*{Proof of Lemma~4.2}.

\begin{lemma} \label{lem:UnionOfBallIsDecomp}
    Let $G$ be a graph, and let $P$ be a path that is a shortest path in $G$ between its endvertices.
    Given~$r \in \N$ we write~$B_p := B_G(p, r)$ for~$p \in V(P)$ and set
    \begin{equation*}
        V_p := \bigcup_{p' \in B_P(p, r+1)} B_{p'}.
    \end{equation*}
    Then $\cH = (P, (V_p)_{p \in P})$ is a partial graph-decomposition of $G$ with support $G' = G[\bigcup_{p \in P} B_p]$ of outer-radial width at most $2r + 1$ and radial spread at most $2r+1$. Moreover, $V_p \subseteq B_G(p, 2r+1)$ for every $p \in V(P)$.
\end{lemma}

\begin{lemma}\label{lem:PathDecomp}
    Let $B$ be a ball in a graph $G$ around a vertex $w$, and let $C$ be a component of $G-B$ whose boundary $\partial_G C$ contains two vertices that are at least $42K+1$ apart for $K \in \N_{\geq 1}$.

    If $G$ has no $K$-fat $K_4^-$ minor, 
    then there is an induced subgraph $Y$ of $G[C,1]$  which admits an honest $(28K+1, 28K+1)$-radial decomposition $(P, \cV)$ modelled on some path~$P = p_0 \dots p_n$ such that
    \begin{enumerate}
        \item \label{itm:PathDecomp:NhoodInFirstAndLast} $N_G(C) \cup \partial_G C \subseteq V_{p_0} \cup V_{p_n}$, and $N_G(C)$ is disjoint from every $V_p$ with $p \neq p_0, p_n$, and
        \item \label{itm:PathDecomp:Admissable} for every component $C'$ of $G-Y$ which meets, or equivalently is contained in, $C$, its neighbourhood $N_G(C')$ is contained in some bag $V_p$ of $(P, \cV)$.
    \end{enumerate}
\end{lemma}

\begin{figure}
    \centering
    \pdfOrNot{\def\svgwidth{0.4\columnwidth} 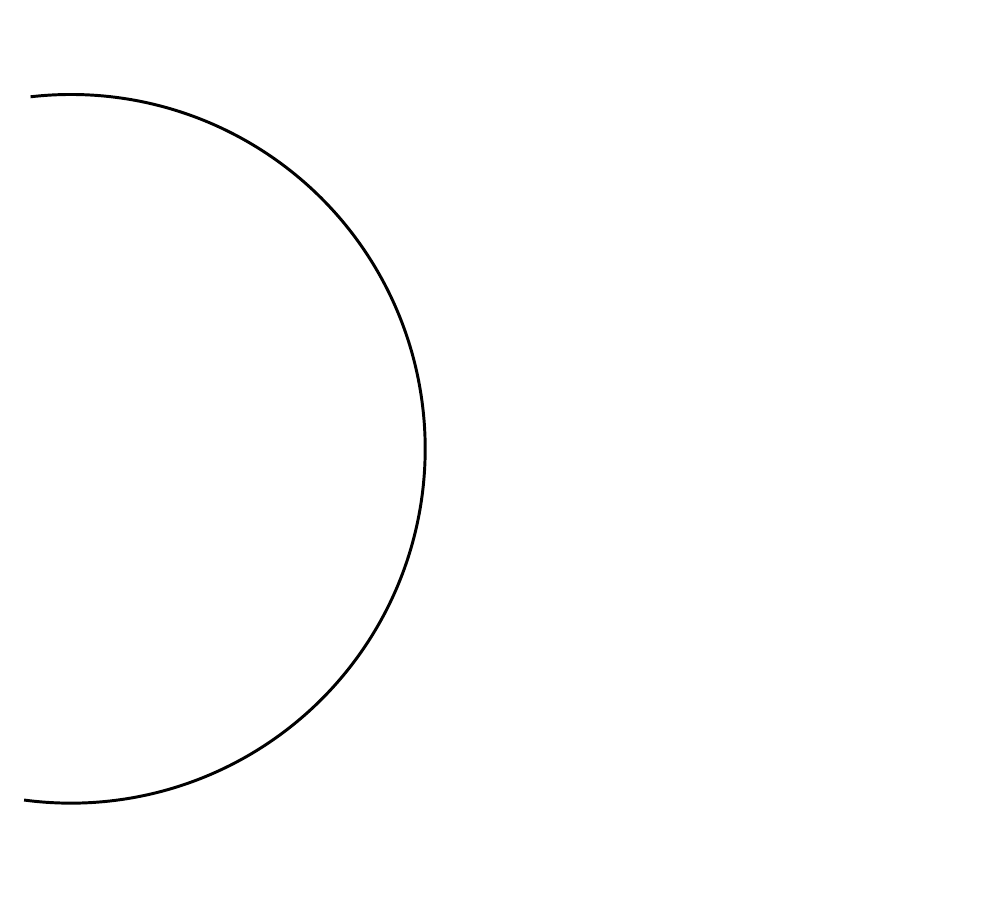}{\includesvg[width=0.4\columnwidth]{svg/pathdecomp.svg}}
    \caption{Depicted in green is the path-decomposition $(P,\cV)$ from \cref{lem:PathDecomp} along the path $P$ which is depicted in blue. By \cref{itm:PathDecomp:Admissable}, the neighbourhood of every component $C'$ of $C-Y$ is contained in some bag in $\cV$.}
    \label{fig:PathDecomp:1}
\end{figure}

\begin{proof}
    Let us assume that $G$ has no $K$-fat $K_4^-$ minor and show that $G$ then has a partial path-decomposition as claimed.
    Let $v_1, v_2$ be two vertices in $\partial_G C$ which are at least $42K+1$ apart.
    Then the balls $B_1$ and $B_2$ in~$G$ of radius $21K-1 \leq \lfloor \dist_G(v_1,v_2)/2 \rfloor - 1$ around $v_1$ and $v_2$, respectively, are disjoint and also there is no $B_1$--$B_2$ edge in $G$.
    In particular, the sets $N_i := B_G(v_i, 7K-1) \cap \partial_G C \subseteq B_i$ for $i \in [2]$ are at least $28K+3$ apart because for all $n_1 \in N_1, n_2 \in N_2$ we have
    \begin{align} \label{eq:distanceN1N2}
    \dist_G(n_1,n_2) \geq \dist_G(v_1,v_2) - \dist_G(n_1,v_1) - \dist_G(v_2,n_2) \geq  (42K+1) - 2 \cdot (7K-1) = 28K+3.
    \end{align}
    Further, since $G$ has no $K$-fat $K_4^-$ minor, \cref{lem:FindingAK4MinusKFatMinor} yields $N_1 \cup N_2 = \partial_G C$.
    
    Let $P = p_0 \dots p_n$ be a shortest $N_1$--$N_2$ path in $G' := G[B_1 \cup V(C) \cup B_2]$.
    Since $B_1$ and $B_2$ are disjoint and not joined by an edge, $N_i$, for $i \in [2]$, separates $B_i \setminus V(C)$ and $N_{3-i} \subseteq V(C)$ in $G'$.
    Hence, every $N_1$--$N_2$ path in $G'$, in particular $P$, is contained in $C$.
    Moreover, for $r_P := 14K$, we have 
    \begin{align}
    \label{eq:ballofcomponentinG'}
    B_G(P, r_P) \subseteq B_G(C, r_P) = B_G(N_1, r_P) \cup  V(C) \cup B_G(N_2, r_P)  \subseteq B_1 \cup V(C) \cup B_2 = V(G'),
    \end{align}
    where we used that $N_1 \cup N_2 = \partial_G C$, $N_i \subseteq B_G(v_i, 7K-1)$ and $B_i = B_G(v_i, 21K-1)$.
    In particular, 
    \begin{equation} \label{eq:PathDecomp:BallsAreEqual}
    B_{G'}(p,r_P) = B_G(p,r_P)
    \end{equation}
    for all $p \in P$.
    So by \cref{lem:UnionOfBallIsDecomp}, the pair $(P, \cV')$ given by 
    \[
    V'_{p} := \bigcup_{p' \in B_P(p, r_P+1)} B_G(p', r_P)
    \]
    is an honest \gd\ of $G'[P, r_P] = G[P, r_P]$ of radial width at most $2 r_P +1 = 28K +1$ and radial spread at most $2 r_P + 1 = 28K +1$.
    Moreover, $\dist_G(p, v) \leq 28K + 1$ for all  $p \in V(P)$ and $v \in V'_p$.

    Now set $Y := G[P,r_P] \cap G[C,1]$.
    As we will see in a moment, the restriction of $(P, \cV')$ to $Y$ has already all desired properties, except that it does not satisfy the second part of~\cref{itm:PathDecomp:NhoodInFirstAndLast}.
    To solve this problem, we adapt $(P, \cV')$ by deleting $N_G(C)$ from all bags $V'_p$ of internal nodes $p \neq p_0, p_n$ of~$P$.
    Formally, define the pair $(P, \cV)$ by $V_p := (V'_p \cap V(Y)) \setminus N_G(C)$ for $p \neq p_0, p_n$ and $V_p := V'_p \cap V(Y)$ for $p = p_0,p_n$.
    We claim that $(P, \cV)$ and $Y$ are as desired.

    We first prove that $(P, \cV)$ is indeed a \gd\ of $Y$. Since $(P, \cV')$ is a \gd\ of $G[P,r_P]$, its restriction to $Y \subseteq G[P, r_P]$ is a \gd\ of $Y$. Thus, it suffices to show that every vertex $v \in N_G(C)$ is contained in precisely one of the bags $V_{p_0}', V_{p_n}'$ and that this bag also contains all neighbours of $v$.

    Since $p_0 \in N_1$, we have $N_1 \subseteq B_G(p_0, 7K-1)$. Otherwise, there is some $u \in N_1$ with $\dist_G(u, p_0) \geq 7K$. Since also $\dist_G(u,v_{2}), \dist_G(p_0,v_{2}) \geq \dist_G(N_1, N_2) \geq 7K$ by \cref{eq:distanceN1N2}, \cref{lem:FindingAK4MinusKFatMinor} applied to $u,p_0,v_2 \in \partial_G C$ yields a $K$-fat $K_4^-$ minor in $G$, which is a contradiction. Analogously, we find $N_2 \subseteq B_G(p_n, 7K-1)$.
    Since $\dist_G(p_0, N_2), \dist_G(p_n, N_1) \geq 28K+3$ by \cref{eq:distanceN1N2} but $\dist_G(p,v) \leq 28K+1$ for all $p \in V(P)$ and $v \in V_p'$, the set $B_G(N_1,1)$ is disjoint from $V_{p_n}'$, and $B_G(N_2,1)$ is disjoint from $V_{p_0}'$.

    Further, since $(P, \cV')$ has outer-radial width and radial spread at most $28K+1$, so does $(P, \cV)$.
    It remains to show that $(P, \cV)$ satisfies \cref{itm:PathDecomp:Admissable}.
    Suppose for a contradiction that \cref{itm:PathDecomp:Admissable} does not hold, that is, there is a component $C'$ of $G-Y$ which is contained in $C$ and which contains $z_{1} \in N_G(B_G(p_{j_1},r_P))$ and $z_{2} \in N_G(B_G(p_{j_2},r_P))$ with $0 \leq j_1 \leq j_2 \leq n$ such that $j_2 - j_1 > 2 \cdot (r_P + 1)$.

    \begin{figure}[ht]
        \centering
        \begin{subfigure}{0.45\linewidth}
                \centering
                \pdfOrNot{\def\svgwidth{\columnwidth} 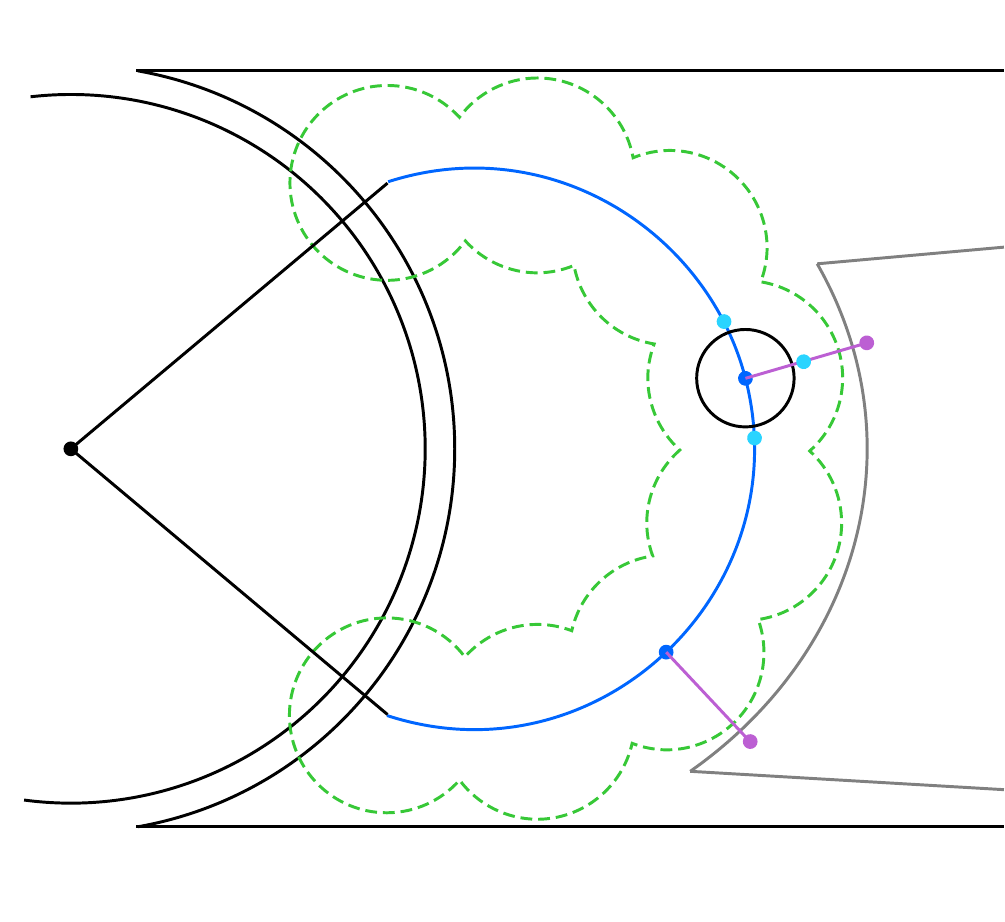}{\includesvg[width=\columnwidth]{svg/insidecomp.svg}}
            \caption{Case 1: $j_1 \geq 7K$ and thus $u_1 \in V(P)$}
            \label{fig:PathDecomp:2a}
        \end{subfigure}
        \hspace{8mm}
        \begin{subfigure}{0.45\linewidth}
            \centering
            \pdfOrNot{\def\svgwidth{0.78\columnwidth} 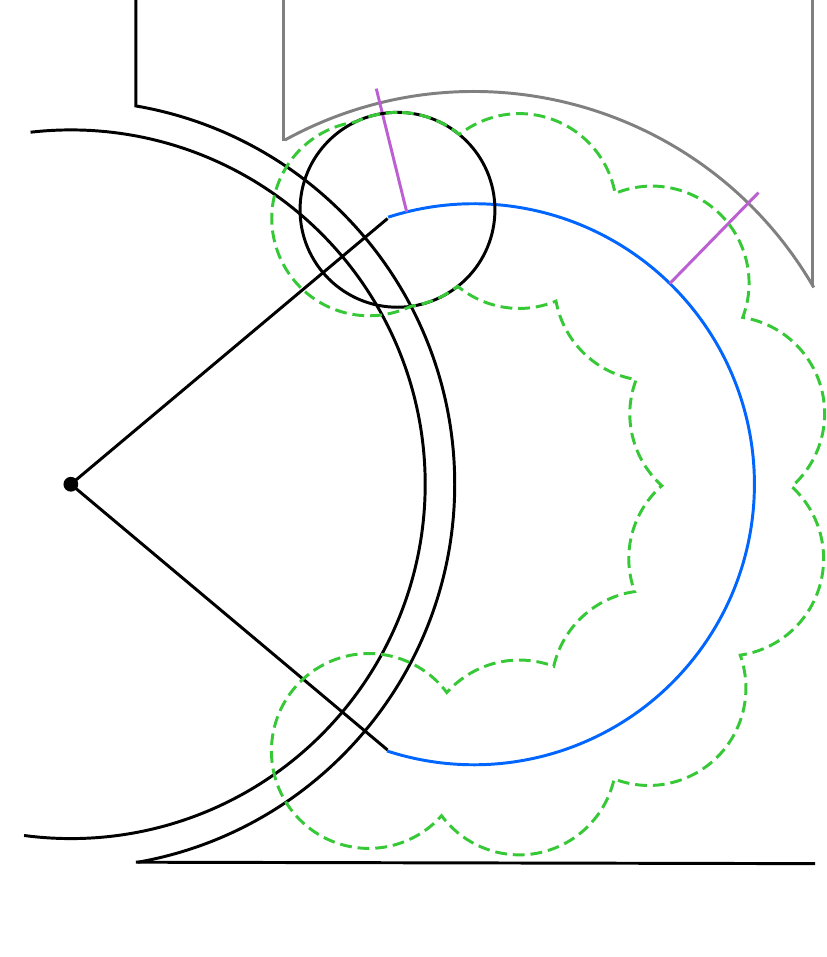}{\includesvg[width=0.78\columnwidth]{svg/closecomp.svg}}
            \caption{Case 2: $j_1 < 7K$ and thus $u_1 \in V(W_1)$}
            \label{fig:PathDecomp:2b}
        \end{subfigure}
        \caption{The situation in the proof of \cref{lem:PathDecomp} if $(P, \cV)$ does not satisfy \cref{itm:PathDecomp:Admissable}.}
        \label{fig:ifpathdecompisnotadmissable}
    \end{figure}

    To derive a contradiction, we find a $K$-fat $K_4^-$ minor in $G$ as follows. Let $Q^i = q^i_0 \dots q^i_{r_P+1}$ be a shortest $p_{j_i}$--$z_i$ path in $G$ and let $W^i = w^i_0 \dots w^i_{r+1}$ be a shortest $w$--$v_i'$ path in $G$ for $v_1' = p_0$ and $v_2'= p_n$. Note that by definition of $z_i$ and $p_{j_i}$, the $Q^i$ have length $r_P + 1$ and are hence shortest $P$--$C'$ paths in $G$. Further, recall 
    that the radius $r$ of the ball $B$ is at least $\dist_G(v_1, v_2) /2 \geq 14K+1$.
    
    Now first assume that $j_1 \geq 7K$ (see \cref{fig:PathDecomp:2a}). Set $B' := B_G(p_{j_1}, 7K)$, and consider $u_1 := p_{j_1-7K-1}$, $u_2 := p_{j_1+7K+1}$ and $u_3 := q^1_{7K+1}$.
    Since $P$ is a shortest path in $G'$ and $Q^1$ is also a shortest $P$--$C'$ path in $G$, it follows from \cref{eq:PathDecomp:BallsAreEqual} for $p_{j_1}$ that $u_1, u_2, u_3 \in N_G(B')$.
    A similar reasoning applied to $u_1$ and $u_2$ yields that the $u_i$ have pairwise distance at least $7K$.
    We claim that $u_1, u_2, u_3$ lie in the same component of $G-B'$, which by \cref{lem:FindingAK4MinusKFatMinor} yields a $K$-fat $K_4^-$ minor in $G$ -- a contradiction.
    Indeed, $Z := p_0Pu_1 \cup W^1 \cup W^2 \cup u_2Pp_n \cup Q^2 \cup C' \cup u_3Q^1z_1$ is connected and contains $u_1, u_2, u_3$.
    Moreover, it follows from a similar reasoning as above together with $j_2 - j_1 > 2 (r_P+1)$, also $Q^2$ being a shortest $P$--$C'$ path in $G$ that $Z$ is disjoint from~$B'$, and $W^1,W^2 \subseteq G[B,1]$.

    Second, assume that $j_1 < 7K$ (see \cref{fig:PathDecomp:2b}). Set $B' := B_G(p_{j_1}, 14K)$, and consider $u_1 := w^1_{r-14K+j_1}$, $u_2 := p_{j_1+14K+1}$ and $u_3 := z_1$. As above, $u_1, u_2, u_3 \in N_G(B')$ and the $u_i$ have pairwise distance at least $7K$.
    Again, we obtain a contradiction from \cref{lem:FindingAK4MinusKFatMinor}, since $u_1, u_2, u_3$ lie in the same component of $G_B'$. Indeed, $Z := W^1u_1 \cup W^2 \cup u_2Pp_n \cup Q^2 \cup C'$ is connected and contains $u_1, u_2, u_3$, and 
    one may follow a similar reasoning as above together with $\dist_G(p_0, p_n) > 2 \cdot (14K+1)$ by \cref{eq:distanceN1N2} to show that $Z$ is also disjoint from $B'$.
\end{proof}

We are now in a position to prove \cref{thm:RadialCactusWidth}. For this, by \cref{thm:CompExtensionLem}, it suffices to show that~$K_4^-$ has property \nameref{tag:star}.

\begin{lemma}\label{lem:K_4^-SatisfiesPartialExtension}
    $K_4^-$ has property \textup{\nameref{tag:star}} (with respect to $R(K) := f'_0(K) := 42K+1$, $f'_1(K) := 28K+1$ and $f''_1(K) = 1$).
\end{lemma}

\begin{proof}
    Let $K \in \N_{\geq 1}$, and let $G$ be a graph with no $K$-fat $K_4^-$ minor.
    Let $B$ be a ball in $G$ of radius $\leq 42K+1$, and let $C$ be a component of $G-B$.

    If every two vertices in $\partial_G C$ are at most $42K$ apart, then the desired honest $(42K+1)$-component-feasible $(42K+1,1)$-radial \gd\ $(H^C, \cV^C) := (H,\cV)$ of $Y^C := Y := G[N_G(C) \cup \partial_G C]$ is given by defining $H$ as a $K_2$ on two vertices $h, h'$ and setting $V_{h} := N_G(C), V_{h'} := N_G(C) \cup \partial_G C$.
    
    Otherwise, we apply \cref{lem:PathDecomp} to the ball $B$ and the component $C$ to obtain an honest decomposition $(P, \cV)$ of an induced subgraph $Y$ of $G[C,1]$, modelled on some path $P = p_0 \dots p_n$ such that $(P, \cV)$ has outer-radial width at most $28K+1$, has radial spread at most $28K+1$, and satisfies \cref{itm:PathDecomp:NhoodInFirstAndLast} and \cref{itm:PathDecomp:Admissable}.
    We obtain the desired honest $(28K+1)$-component-feasible $(28K+1, 28K+1)$-radial \gd\ $(H^C, \cV^C) := (H, \cV)$ of $Y^C := Y$ by adding a node $h$ and edges $hp_0, hp_n$ to $P$ and setting $V_{h} := N_G(C)$. Indeed,~\cref{itm:PathDecomp:NhoodInFirstAndLast} ensures that $(H, \cV)$ is an (honest) \gd\ of $Y$ and the radial spread of each vertex in $N_G(C)$ in $(H, \cV)$ is at most $1$.
    The radial spread of every other vertex in $Y$ did not change, that is, it is still at most $28K+1$.
    Since $V_h = N_G(C) \subseteq B$, we have $\rad_G(V_h) \leq \rad_G(B) \leq 42K+1$, and thus $(H, \cV)$ has outer-radial width at most $42K+1$.
\end{proof}

\begin{proof}[Proof of \cref{thm:RadialCactusWidth}]
     By \cref{thm:CompExtensionLem}, \cref{lem:K_4^-SatisfiesPartialExtension} immediately yields \cref{thm:RadialCactusWidth}.
\end{proof}

\begin{proof}[Proof of \cref{main:Cactus}]
    Let $G$ be a graph with no $K$-fat $K_4^-$ minor.
    By \cref{lem:GraphDecToQuasiIso}, \cref{thm:RadialCactusWidth} yields that there exists a graph $H$ with no $K_4^-$ minor which is $(84K+2, 84K+2)$-quasi-isometric to $G$.
    Hence, it follows from \cref{lem:inversequasiisom} that $G$ is $(84K+2, 3 \cdot (84K+2)^2)$-quasi-isometric to $H$.
    Thus, \cref{main:Cactus} holds with $f(K) := (84K+2, 3 \cdot (84K+2)^2)$.
\end{proof}

\section{Proof of \texorpdfstring{\cref{main:FatK4}}{Theorem 2}: \texorpdfstring{\cref{conj:agelospanos}}{Conjecture 1} for \texorpdfstring{$X = K_4$}{X = K4}} \label{sec:SeriesParallel}

In this section we prove \cref{main:FatK4}, which we restate here in the terminology of \gd s.

\begin{customthm}{\cref*{main:FatK4}$^\prime$} \label{thm:RadialSeriesParallelWidth}
    Every graph $G$ with no $K$-fat $K_4$ minor for $K \in \N_{\geq 1}$ admits an honest $(25235K + 71, 22)$-radial decomposition $(H, \cV)$ modelled on a graph $H$ with no $K_4$ minor.
\end{customthm}

\noindent Note that by \cref{lem:GraphDecToQuasiIso,lem:inversequasiisom}, \cref{thm:RadialSeriesParallelWidth} immediately implies \cref{main:FatK4}.
\medskip

\defn{Two-terminal graphs} are graphs with two distinguished (not necessarily distinct) vertices $h_1,h_2$, its \defn{terminals}, which we refer to as \defn{source} and \defn{sink}. We denote by $\cH_{SP}$ the class of all two-terminal graphs $H$ with terminals $h_1,h_2 \in V(H)$ that satisfy $H + h_1h_2 \in \Forbminor(K_4)$.\footnote{Note that we allow $h_1 = h_2$, in which case $h_1h_2$ is a loop. Thus, we have for two-terminal graphs $H$ with $h_1 = h_2$ that $H \in \cH_{SP}$ if and only if $H \in \Forbminor(K_4)$.}
We remark that the finite $2$-connected graphs in $\cH_{SP}$ are precisely the $2$-connected series-parallel graphs.

A \defn{parallel composition} of two-terminal graphs is obtained by identifying the sources and identifying the sinks. 
A \defn{series composition} of a pair of two-terminal graphs is obtained by identifying the sink of one of them with the source of the other one.
Since $K_4$ is $3$-connected, the following facts about $\cH_{SP}$ follow easily.

\begin{proposition}\label{obs:ConstrOfSPGraphs}
    $\cH_{SP}$ is closed under the following operations:
    \begin{itemize}
        \item  parallel and series composition,
        \item  subdividing edges,
        \item  adding a path of length at least $2$ between any two adjacent vertices, and
        \item (infinite) union over graphs $H_0 \subseteq H_1 \subseteq \ldots \in \cH_{SP}$.
    \end{itemize}
\end{proposition}

\begin{proof}
    Since $K_4$ is $3$-connected, the first three assertions follow easily. For the fourth assertion, note that since $K_4$ is finite, $H + h_1h_2$, with $H := \bigcup_{i \in \N} H_i$, contains a model of $K_4$ if and only if it contains one with finite branch sets. By the definition of $H$ as the union of the $H^i \in \cH_{SP}$, this implies that $H+h_1h_2$ has no $K_4$ minor, and hence $H \in \cH_{SP}$. 
\end{proof}

Throughout this section we fix 
    \begin{align*}
        R_0(K) &:= 129\cdot 5K + 5K = 130 \cdot 5K,\\ 
        R'_0(K) &:= 3R_0(K) + 5K + 1,\\ 
        \ell(K) &:= 2R_0(K) + 5K + 2,\\ 
        R_1(K) &:= 4\cdot (2(\ell(K) +22K+1) + 11K +2), \text{ and}\\ 
        R_2(K) &:= 2R_1(K) + 5K + 3. 
    \end{align*}

Let us briefly sketch the proof of \cref{thm:RadialSeriesParallelWidth}. For this, let $K \in \N_{\geq 1}$, and let $G$ be a graph with no $K$-fat $K_4$ minor. By \cref{thm:CompExtensionLem} (with $R(K) := R_2(K)$, $f'_0(K) := R_2(K) + 2R'_0(K)+2$ and $f'_1(K) = f''_1(K) := 7$), and because $R_2(K) + 2R'_0(K) + 2 = 25235K + 71$ and $7 + 2\cdot 7 + 1 = 22$, it suffices to show that $K_4$ satisfies property \textup{\nameref{tag:star}}, that is for every ball $B$ in $G$ of radius $\leq R_2(K)$ and every component $C$ of $G-B$ there exists a $R_2(K)$-component-feasible, $(R_2(K)+2R'_0(K)+2, 7)$-radial partial decomposition $\cH^C$ of $G$ modelled on a graph $H \in \Forbminor(K_4)$ with support $Y^C \subseteq G[C,1]$ such that $\partial_G C \subseteq V(Y^C)$ (see \cref{lem:K_4SatisfiesPartialExtension}). So let such a ball $B$ and component $C$ of $G-B$ be given.
To define $\cH^C$, we distinguish two cases. 
If the boundary $\partial_G D$ of some component $D$ of $C - B_G(B, 22K+1)$ contains three vertices that are pairwise far apart, then we obtain the desired partial decomposition with support $Y^C \subseteq G[C,1]$ by applying the following \cref{lem:CompWithThreeVertices} to $C$ and $C^* := D$. 

\begin{figure}
    \centering
    \begin{subfigure}{0.45\linewidth}
        \pdfOrNot{\def\svgwidth{\columnwidth} 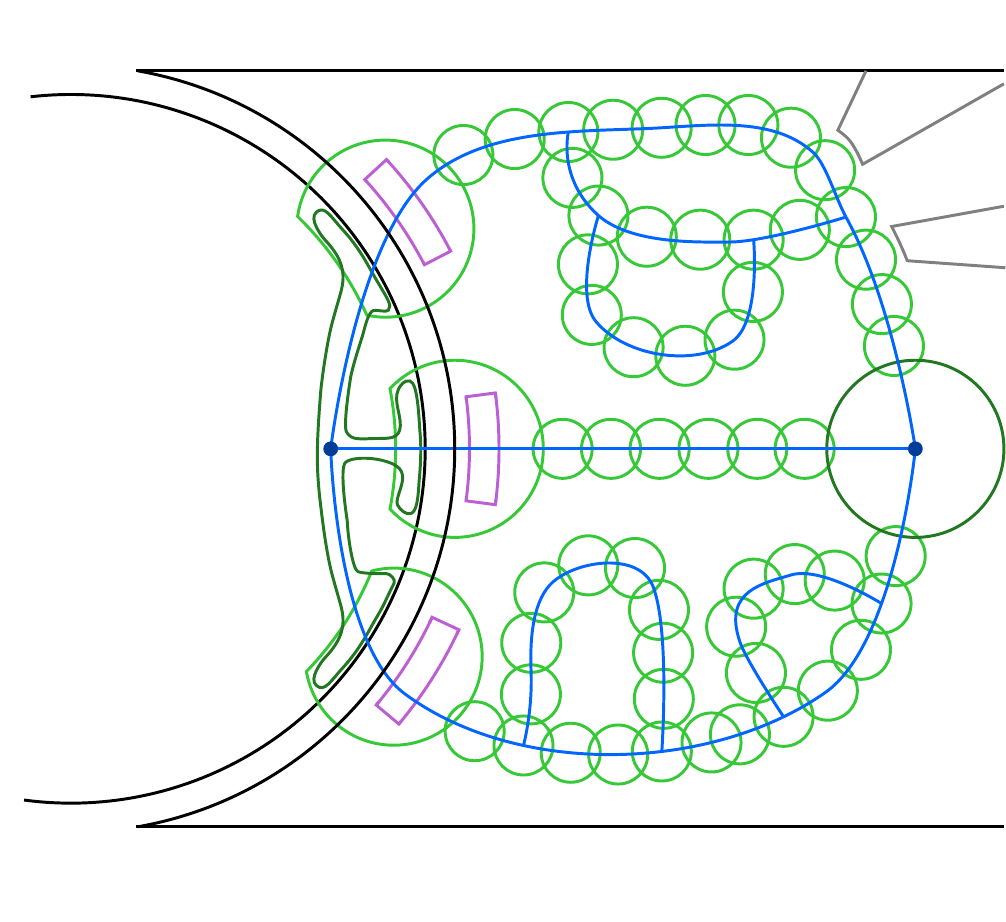}{\includesvg[width=\columnwidth]{svg/seriesparallelcomp2.svg}}
        \caption{\cref{lem:CompWithThreeVertices}}
        \label{fig:CompWidthThreeVertices}
    \end{subfigure}
    \hspace{8mm}
    \begin{subfigure}{0.45\linewidth}
        \pdfOrNot{\def\svgwidth{\columnwidth} 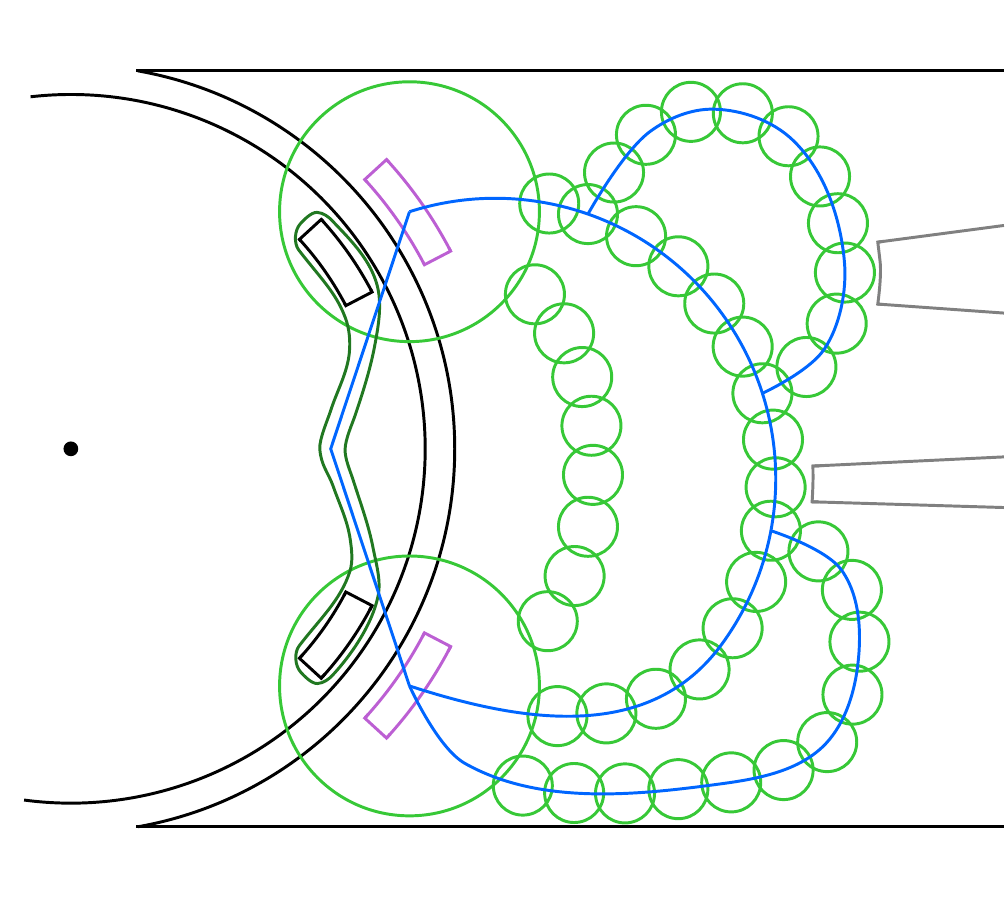}{\includesvg[width=\columnwidth]{svg/seriesparallelcomp.svg}}
        \caption{\cref{lem:CompWithBipOfNhood}}
        \label{fig:CompWidthBipartOfNhood}
    \end{subfigure}
    \caption{Depicted are the \gd s $(H, \cV)$ with support $Y$ in \cref{lem:CompWithBipOfNhood,lem:CompWithThreeVertices} if $G$ has no $K$-fat $K_4$ minor. By \cref{lem:CompWithThreeVertices}~\cref{itm:CompWithThreeVertices:CompsNhood} and \cref{lem:CompWithBipOfNhood}~\cref{itm:CompWithBipOfNhood:CompsNhood} the neighbourhood of every component of $C-Y$ is contained in some bag in $\cV$.}
    \label{fig:LemmasForK_4}
\end{figure}

\begin{lemma}\label{lem:CompWithThreeVertices}
    Let $K \in \N_{\geq 1}$, let $B$ be a ball in a graph $G$ around a vertex $w$ of radius $r \in \N_{\geq 1}$, and let $C$ be a component of $G-B$. Suppose there is a component $C^*$ of $C - B_G(B,22K+1)$ such that $\partial_G C^*$ contains three vertices that are pairwise at least $R_1(K)$ apart. 
    
    If $G$ has no $K$-fat $K_4$ minor, then there exists an induced subgraph $Y$ of $G[C,1]$ with $\partial_G C \subseteq V(Y)$ which admits an honest decomposition $(H, \cV)$ modelled on some graph $H \in \Forbminor(K_4)$ such that
    \begin{enumerate}
        \item \label{itm:CompWithThreeVertices:FirstBag} $(H, \cV)$ contains $N_G(C)$ as a bag,
        \item \label{itm:CompWithThreeVertices:Width} $\orw((H, \cV)) \leq \max\{r + 2R'_0(K)+2, R_2(K)\}$ and 
        \item \label{itm:CompWithThreeVertices:Spread} $\irs((H, \cV)) \leq 7$, and
        \item \label{itm:CompWithThreeVertices:CompsNhood} 
        $(H, \cV)$ is $R$-component-feasible for $R := \max\{r, R_2(K)\}$.
    \end{enumerate}
\end{lemma}

Otherwise, we define $\cH^C$ by considering each component of $C- B_G(B,22K+1)$ separately. For this, we first take a $K_2$ on vertices $g, g'$ and set $V_g := N_G(C)$ and $V_{g'} := (B_G(B, 22K+1) \cap V(C)) \cup N_G(C)$. Next, to ensure that $\cH^C$ is $R_2(K)$-component-feasible, we extend this decomposition into the components $D$ of $C-B_G(B,22K+1)$ as follows.
If some $\partial_G D$ has radius $\leq R_2(K)-1$, then we may simply take a $K_2$ on vertices $g^D,h^D$, set $V_{g^D} := N_G(D) =: V_{h^D}$ and identify $g^D$ with $g'$. Otherwise, if $\partial_G D$ can be partitioned into two sets $N_0, N_1$ of radius at most $R_1(K)$ that are at least $5K+2$ far apart, then we obtain a partial decomposition $(H^C, \cV^C)$ with support $Y^C \subseteq G[C,1]$ modelled on a graph with no $K_4$ minor by the next \cref{lem:CompWithBipOfNhood} (applied to $C :=  D$), and we can then identify the node of $H^D$ given by \cref{lem:CompWithBipOfNhood}~\cref{itm:CompWithBipOfNhood:FirstBag} with $g'$. This then yields the desired partial decomposition $\cH^C$.

\begin{lemma}\label{lem:CompWithBipOfNhood}
    Let $K \in \N_{\geq 1}$, and let $B$ be a ball in a graph $G$ around a vertex $w$ of radius $r \in \N$, and let $C$ be a component of $G - B$. Suppose there are vertices $v_1, v_2 \in \partial_G C$ which are at least $2 R_1(K) + 5K + 2$ apart and $\partial_G C \subseteq B_G(v_1, R_1(K)) \cup B_G(v_2, R_1(K))$. 
    
    If $G$ has no $K$-fat~$K_4$ minor, then there exists an induced subgraph $Y$ of $G[C,1]$ with $\partial_G C \subseteq V(Y)$ which admits an honest decomposition $(H, \cV)$ modelled on some graph $H \in \Forbminor(K_4)$ such that
    \begin{enumerate}[label=\rm{(\roman*)}]
        \item \label{itm:CompWithBipOfNhood:FirstBag} $(H,\cV)$ contains $N_G(C)$ as a bag,
        \item \label{itm:CompWithBipOfNhood:Radius} $\orw((H, \cV)) \leq R := \max\{r, R_1(K) + 2R_0(K) + 5K + 2\}$,
        \item \label{itm:CompWithBipOfNhood:Spread} $\irs((H, \cV)) \leq 3$, and
        \item \label{itm:CompWithBipOfNhood:CompsNhood} $(H, \cV)$ is $R$-component-feasible; moreover, for every component of $D$ of $G-Y$ which meets, or equivalently is contained in, $C$, its neighbourhood $N_G(D)$ is contained in some bag $V_h$ of $(H, \cV)$ with $\rad(V_h) \leq R_1(K)+2R_0(K)+5K+1$.
    \end{enumerate}
\end{lemma}

Let us now deduce \cref{thm:RadialSeriesParallelWidth} from \cref{lem:CompWithBipOfNhood,lem:CompWithThreeVertices}. For this, as noted earlier, it suffices to prove that $K_4$ has property \nameref{tag:star}.

\begin{lemma}\label{lem:K_4SatisfiesPartialExtension}
    $K_4$ has property \textup{\nameref{tag:star}} (with respect to $R(K) = R_2(K)$, $f'_0(K) := R_2(K)+2R'_0(K) + 2$ and $f'_1(K) = f''_1(K) := 7$).
\end{lemma}

\begin{proof}
    Let $K \in \N_{\geq 1}$, and let $G$ be a graph with no $K$-fat $K_4$ minor.
    Let $B$ be a ball in $G$ of radius $r \leq R_2(K)$ around a vertex $v \in V(G)$, and let $C$ be a component of $G-B$. 
    Further, let $\cD$ be the set of components of $C-B_G(B,22K+1)$, and first assume that there is some $D \in \cD$ such that $\partial_G D$ contains three vertices that are pairwise at least $R_1(K)$ apart. Then we obtain an induced subgraph $Y \subseteq G[C,1]$ with $\partial_G C \subseteq V(Y)$ and the desired honest $R_2(K)$-component-feasible $(f'_0(K), 7)$-partial decomposition $(H, \cV)$ of $Y$ with $N_G(C) = V_g$ for some $g \in V(H)$ by applying \cref{lem:CompWithThreeVertices} to the ball $B$ of radius $r \leq R_2(K)$ and the component $C$ of $G-B$, which completes the proof in this case.
    
    Otherwise, let $\cD' \subseteq \cD$ be the set of components $D$ of $C-B_G(B,22K+1)$ such that there are two vertices in $\partial_G D$ which are at least $2 R_1(K) + 5K + 2$ apart. Then by assumption, every $D \in \cD'$ satisfies the premise of \cref{lem:CompWithBipOfNhood}; let $(H^D, \cV^D)$ be the partial decomposition with support $Y^D$ obtained from applying \cref{lem:CompWithBipOfNhood} to the ball $B_G(B, 22K+1)$, which is a ball of radius $r+22K+1$ around $v$, and the component $D$. Let $g^D$ be a node of $H^D$ with $V^D_{g^D} = N_G(D)$, which exists by \cref{itm:CompWithBipOfNhood:FirstBag}.
    Further, for every component $D \in \cD \setminus \cD'$, let $H^D$ be a $K_2$ on vertices $g^D, h^D$ and set $V^D_{g^D} :=  V^D_{h^D} := N_G(D)$. 
    
    Now let $H$ be the graph obtained from the disjoint union of the $H^D$ for $D \in \cD'$ by first identifying all $g^D$ to a single node, which we call $g'$, and adding a node $g$ and the edge $gg'$. Note that $H \in \Forbminor(K_4)$ since it is the $1$-sum of the $K_4$-minor-free graphs $H^D$ and $K_2$. Set $V_h := V^D_h$ for all $h \in V(H)\setminus \{g,g'\}$ where $D$ is the unique component in $\cD'$ such that $h \in V(H^D)$ and $V_{g'} := (B_G(B,22K+1) \cap V(C)) \cup N_G(C)$ and $V_g := N_G(C)$. Then, by construction, $(H, \cV)$ is a partial decomposition with support $Y := G[\bigcup V(Y^D) \cup V_{g'}]$ since $V^D_{g^D} \subseteq V_{g'}$ for all $D \in \cD'$.
    
    We claim that $(H, \cV)$ is as desired. Indeed, the outer-radial width of $(H,\cV)$ is $R_2(K) + 22K + 1 \leq f_0'(K)$, since $\rad(V_{h^D}),\rad_G(V_g) \leq \rad_G(V_{g'}) \leq \rad_G(B) + 22K+1 \leq R_2(K) + 22K + 1$ for every $D \in \cD \setminus \cD'$ and $\rad(V_h) \leq R_2(K) + 22K+1$ for all other $h \in V(H)$ by \cref{itm:CompWithBipOfNhood:Radius} of \cref{lem:CompWithBipOfNhood}.
    Note that the radial spread of $(H, \cV)$ is $\leq 2 \cdot \max_{D \in \cD} \irs((H^D, \cV^D)) \leq 6 \leq f_1'(K)$, since $V_{g} \subseteq V_{g'}$, or it is $1$ if $\cD = \emptyset$.
    
    By construction, it follows immediately from \cref{lem:CompWithBipOfNhood}~\cref{itm:CompWithBipOfNhood:CompsNhood} that for every component of $G-Y$ which meets $C$ there exists a bag $V_h$ containing its neighbourhood. Moreover, clearly, every component of $G-Y$ which does not meet $C$ has only neighbours in $N_G(C) = V_g \subseteq B$. All in all, $(H,\cV)$ is  $R(K)$-component-feasible.
\end{proof}

\begin{proof}[Proof of \cref{thm:RadialSeriesParallelWidth}]
    By \cref{thm:CompExtensionLem}, \cref{lem:K_4SatisfiesPartialExtension} immediately yields \cref{thm:RadialSeriesParallelWidth}.
\end{proof}

\begin{proof}[Proof of \cref{main:FatK4}]
    Let $G$ be a graph with no $K$-fat $K_4$ minor.
    By \cref{lem:GraphDecToQuasiIso}, \cref{thm:RadialSeriesParallelWidth} yields that there exists a graph $H$ with no $K_4$ minor which is $(50470K + 142, 50470K + 142)$-quasi-isometric to $G$.
    Hence, it follows from \cref{lem:inversequasiisom} that $G$ is $(50470K + 142, 3 \cdot (50470K + 142)^2)$-quasi-isometric to $H$.
    Thus, \cref{main:Cactus} holds with $f(K) := (50470K + 142, 3 \cdot(50470K + 142)^2)$.
\end{proof}

The remainder of this section is devoted to the proofs of \cref{lem:CompWithBipOfNhood,lem:CompWithThreeVertices}. For this, we first show in \cref{subsec:FindingAKFatK4Minot} an auxiliary lemma that we will use to find $K_4$ as a $K$-fat minor in a graph $G$ whenever we cannot find our desired decomposition of $G$ modelled on a graph with no $K_4$ minor of small outer-radial width. Then, in \cref{subsec:Lemma1,subsec:Lemma2}, we prove \cref{lem:CompWithBipOfNhood,lem:CompWithThreeVertices}, respectively.

\subsection{Finding a \texorpdfstring{$K$}{K}-fat \texorpdfstring{$K_4$}{K4} minor} \label{subsec:FindingAKFatK4Minot}

\begin{lemma}\label{lem:FindingAFatK4}
    Let $G$ be a graph, $K, r_1, r_2 \in \N_{\geq 1}$, and let $v_1, v_2$ be two vertices of $G$ that are at least $r_1 + r_2 + 5K$ apart. Set $B_i := B_G(v_i, r_i)$ for $i \in [2]$, and suppose there are three $B_1$--$B_2$ paths $P_1$, $P_2$ and $P_3$ which are pairwise at least $5K$ apart.
    
    If $G$ has no $K$-fat $K_4$ minor, then the $\mathring{P_i}$ lie in distinct components of $G - (B_1 \cup B_2)$.
\end{lemma}

\begin{proof}
    Let us first note that since the paths $P_1,P_2,P_3$ are pairwise at least $5K$ apart, also $r_i \geq \lceil5K/2\rceil > 2K$ for $i \in [2]$. 
    For $i \in [3]$, let $Q^i := q^i_0 \dots q^i_{n_i}$ be a shortest $B_G(v_1, r_1-2K)$--$B_G(v_2, r_2-2K)$ path in $G[B_1] \cup P_i \cup G[B_2]$; in particular, $n_i := ||P_i|| + 4K$.

    Suppose that at least two $\mathring{P_i}$ lie in the same component $C$ of $G - (B_1 \cup B_2)$. We show that $G$ has a $K$-fat~$K_4$ minor. For this, assume without loss of generality that $P_1 \subseteq G[C,1]$. Let~$W$ be a $P_1$--$B_G(Q^2 \cup Q^3, K)$ path in $C$; by symmetry, we may assume that $W$ ends in $B_G(Q^2, K)$.
    Further, let $W' = w_0 \dots w_m$ be a shortest $P_1$--$Q^2$ path in $G[Q^2, K] \cup W$. Note that $W'$ ends in a vertex in $q^2_KQ^2q^2_{n_2-K} = Q^2 \cap G[P_2, K]$, as $W \subseteq C \subseteq G - (B_1 \cup B_2)$ . In particular, since $Q^2$ is a shortest $B_G(v_1, r_1-2K)$--$B_G(v_2, r_2-2K)$ path in $G[B_1 \cup V(P_2) \cup B_2]$ it follows that $w_{m-K}W'w_m \subseteq G[P_2,2K]$.

    \begin{figure}[ht]
        \centering
        \pdfOrNot{\def\svgwidth{0.6\columnwidth} 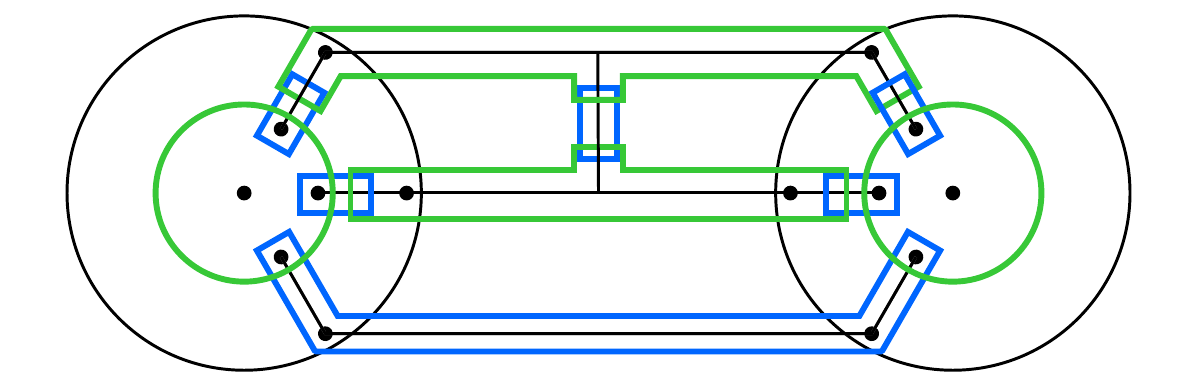}{\includesvg[width=0.6\columnwidth]{svg/k4.svg}}
        \caption{The $K$-fat $K_4$ minor in the proof of \cref{lem:FindingAFatK4}.}
        \label{fig:FindingAFatK4}
    \end{figure}
    
    By defining the branch sets as
    \begin{align*}
    V_1 &:= V(q^1_{K}Q^1q^1_{n_1 - K}) \cup V(w_0W'w_K),\; V_2 := V(q^2_{K}Q^2q^2_{n_2 - K} w_{m-K}W'w_m),\\
    V_3 &:= B_G(v_1, r_1-2K) \text{ and } V_4 := B_G(v_2, r_2-2K).
    \end{align*}
    and the branch paths as
    \[
    E_{12} := w_KW'w_{m-K},\; E_{i3} := q^i_{0}Q^iq_{K}^i,\; E_{i4} := q^i_{n_i - K}Q^iq^i_{n_i} \text{ and } E_{34} := Q^3
    \]
    for $i \in [2]$ we obtain a $K$-fat $K_4$-model in $G$ (see \cref{fig:FindingAFatK4}).
    Indeed, the sets $V_i$ and the paths $E_{ij}$ obviously form a model of $K_4$. It remains to show that this $K_4$-model is $K$-fat.
    
    By construction, we have $V_3, V(E_{i3}) \subseteq B_1$ and $V_4, V(E_{i4}) \subseteq B_2$ for $i \in [2]$, and thus by assumption 
    \[
    \dist_G(V_3, V_4), \dist_G(E_{i3}, E_{j4}), \dist_G(V_3, E_{i4}), \dist_G(V_4, E_{i3}) \geq \dist_G(B_1, B_2) \geq 5K \geq K
    \]
    for $i, j \in [2]$. 
    Moreover, $\dist_G(V(Q^1) \cup V_1, V(Q^2) \cup V_2) \geq \dist_G(P_1, P_2) - 2\cdot 2K \geq 5K - 2\cdot 2K = K$ since $V(Q^i) \cup V_i \subseteq B_G(P_i, 2K)$ for $i \in [2]$.
    This yields that 
    \[
    \dist_G(V_1, V_2), \dist_G(V_1, E_{2i}), \dist_G(V_2, E_{1i}), \dist_G(E_{1i}, E_{2i}) \geq K
    \]
    for $i \in \{3,4\}$.
    As above, we have $d_G(Q^3, V(Q^i) \cup V_i) \geq K$ since also $Q^3 \subseteq G[P_3,2K]$ and $\dist_G(P_3, P_i) \geq 5K$ for $i \in [2]$. 
    Further, $\dist_G(Q^3, W) \geq K$ by the choice of $W$.
    Thus,
    \[
    \dist_G(E_{34}, V_i), \dist_G(E_{34}, E_{ij}), \dist_G(E_{34}, E_{12}) \geq K
    \]
    for $i \in [2]$ and $j \in \{3,4\}$. 
    As $C$ is a component of $G-(B_1 \cup B_2)$, we have $\dist_G(C, B_G(v_i, r_i-K)) = K + 1$ for $i \in [2]$. Since $E_{12} \subseteq W \subseteq C$ this implies that
    \[
    \dist_G(E_{12}, E_{ij}), \dist_G(E_{12}, V_j) \geq \dist_G(C, B_G(v_i, r_i-K)) = K + 1 \geq K
    \]
    for $i \in [2]$ and $j \in \{3, 4\}$.
    Finally, we obtain
    \[
    \dist_G(V_i, V_j) \geq \dist_G(B_G(C,K), B_G(v_{j-2}, r_{j-2} - 2K))  = 2K - K = K
    \]
    for $i \in [2]$ and $j \in \{3, 4\}$.
    This completes the proof.
\end{proof}

\subsection{Proof of \texorpdfstring{\cref{lem:CompWithBipOfNhood}}{Lemma 5.2}} \label{subsec:Lemma1}

Let us briefly sketch the proof of \cref{lem:CompWithBipOfNhood}. For this, let us first recall its premises: Let $G$ be a graph with no $K$-fat $K_4$ minor for $K \in \N_{\geq 1}$, let $B$ be a ball in $G$, and let $C$ be a component of $G-B$. Suppose that there are vertices $v_1, v_2 \in \partial_G C$ such that the sets $B_i := (v_i, R_1(K))$ for $i \in [2]$ induce a partition of $\partial_G C$ and such that $\dist_G(B_1, B_2) \geq 5K+2$. We then construct for every component $D$ of $C - (B_1 \cup B_2)$ a partial decomposition $(H^D, \cV^D)$ of $G$ with support $Y \subseteq G[D,1]$ modelled on a graph $H^D \in \cH_{SP}$ with terminals $h^D_1, h^D_2$ such that $V^D_{h^D_i} = N_G(D) \cap B_i$ (\cref{lem:ComponentAttachingToTwoBalls}).
To obtain the desired partial \gd\ $(H, \cV)$ to satisfy the conclusion of \cref{lem:CompWithBipOfNhood}, we then glue all the decompositions $(H^D, \cV^D)$ together. More precisely, we obtain $H$ from the disjoint union of the $H^D$ by first identifying all the $h_1^D$ to a node~$h_1$ as well as all the $h_2^D$ to a node~$h_2$, and then adding a vertex $g$ and the edges $gh_1, gh_2$. As $H$ hence arises as the parallel composition of graphs in $\cH_{SP}$, it is again in $\cH_{SP}$ by \cref{obs:ConstrOfSPGraphs} and hence has no $K_4$ minor. The desired partial \gd\ is then given by $(H, \cV)$ where $V_{h_i} := (B_i \cap V(C)) \cup (N_G(C) \cap B_G(B_i, 1))$, $V_g := N_G(C)$ and $V_h := V^D_h$ for all other $h \in V(H)$ where $H^D$ is the unique graph containing $h$.

In the proof of \cref{lem:ComponentAttachingToTwoBalls}, we construct the partial \gd s $(H^D, \cV^D)$ recursively via \cref{prop:BallMeetingAllPathsInComp} below.
For this, we will need the following coarse version of Menger's theorem for two paths.

\begin{theorem}{\cite{distancemengerfortwo}*{Theorem 2}} \label{thm:DistanceMengerForTwoPaths}
    Let $G$ be a graph, $X, Y \subseteq V(G)$, and let $Q$ be a shortest $X$--$Y$ path in~$G$. For all $d \in \N_{\geq 1}$, either there exist two disjoint $X-Y$ paths $P_1$, $P_2$ that are at least $d$ apart or there exists $z \in V(Q)$ such that $B_G(z, 129 d)$ intersects every $X - Y$ path.\footnote{That the vertex $z$ may be chosen on some shortest $X$--$Y$ path is not stated in \cite{distancemengerfortwo}*{Theorem 1}, but it follows easily from its proof. For convenience, we also remark that if $z$ is any vertex such that $B_G(z, 129d)$ intersects every $X$--$Y$ path, then $B_G(z, 129d)$ in particular intersects $Q$ in some vertex $z'$. Then also $B_G(z', 2\cdot129d)$ intersects every $X$--$Y$ path, so the assertion follows directly from \cite{distancemengerfortwo}*{Theorem 1} if we increase the radius of the ball by a factor of $2$.}
\end{theorem}

\begin{lemma}\label{prop:BallMeetingAllPathsInComp}
    Let $G$ be a graph, $K, r_1, r_2 \in \N_{\geq 1}$, and let $v_1, v_2$ be two vertices of $G$ that are at least $r_1 + r_2 + 5K+2$ apart. Set $B_i := B_G(v_i, r_i)$ for $i \in [2]$. Let~$C$ be a component of $G - (B_1 \cup B_2)$ which attaches to $B_1$ and to $B_2$, and suppose there exists a $B_1$--$B_2$ path $P$ in $G$ such that $\dist_G(P, G[C,1]) \geq 5K$.
    Set $N_i := N_G(C) \cap B_i$ for $i \in [2]$.
    If $G$ has no $K$-fat $K_4$ minor, then there exists a vertex $u \in B_G(C,1)$ such that every $N_1$--$N_2$ path through $C$ intersects $B_G(u, 129\cdot5K)$. 
    
    Moreover, for $B' := B_G(u, 130\cdot 5K)$, for $i \in [2]$ and for every component $C'$ of $G - (B_i \cup B')$ that is contained in $C$ and attaches to $B_i$ and $B'$, there exists a $B_i$--$B'$ path $P'$ such that $d_G(P, G[C',1]) \geq 5K$.
\end{lemma}

\begin{figure}
    \centering
    \pdfOrNot{\def\svgwidth{0.6\columnwidth} 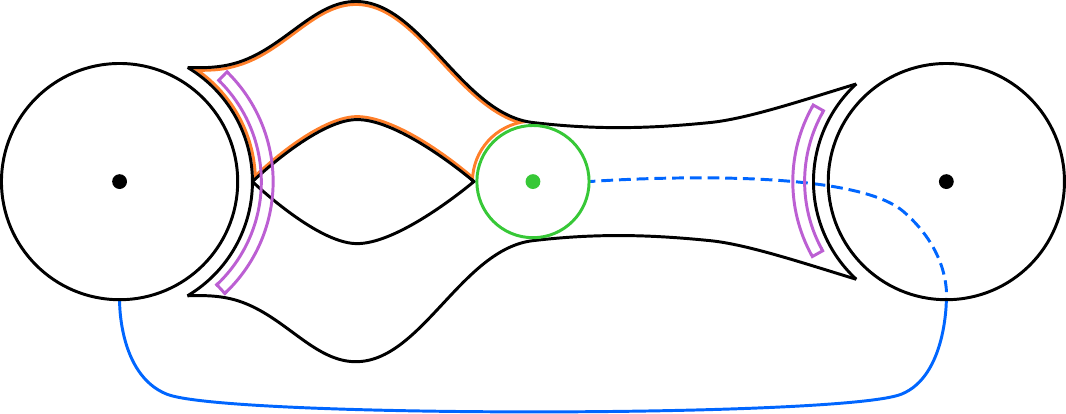}{\includesvg[width=0.6\columnwidth]{svg/ApplicationCoarseMenger.svg}}
    \caption{The situation in \cref{prop:BallMeetingAllPathsInComp}: A component $C$ of $G-(B_1 \cup B_2)$ and the ball $B' \supseteq B$ around $u$. Depicted in blue is the path $P$ and its dashed extension $P'$, as required for the `moreover' part applied to the orange component $C'$.}
    \label{fig:ApplicationCoarseMenger}
\end{figure}

\noindent See \cref{fig:ApplicationCoarseMenger} for a depiction of the situation in \cref{prop:BallMeetingAllPathsInComp}.

\begin{proof}
    Let $Q$ be a shortest $N_1$--$N_2$ path through $C$. 
    As $\dist_G(v_1,v_2) \geq r_1 + r_2 +5K+2$, the balls $B_G(B_1, \lceil(5/2)K\rceil)$ and $B_G(B_2, \lceil(5/2)K\rceil)$ are disjoint and not joined by an edge.
    Thus, $Q$ is a  shortest $N_1$--$N_2$ path in $G' := G[C,1] \cup G[B_1 \cup B_2, \lceil(5/2)K\rceil]$.
    By applying \cref{thm:DistanceMengerForTwoPaths} in $G'$ to the sets $N_1$ and $N_2$ with $d= 5K$, we obtain either a vertex $u$ of $Q \subseteq G[C,1]$ such that every $N_1$--$N_2$ path in $G'$ intersects $B_G(u, 129\cdot5K)$ or two $N_1$--$N_2$ paths $Q_1$ and $Q_2$ in $G'$ such that $\dist_{G'}(Q_1, Q_2) \geq 5K$. 

    In the former case we are done (except for the `moreover'-part) since $u$ lies on a shortest $N_1$--$N_2$ path through $C$ and $B':= B_G(u, 130\cdot5K) \supseteq B_G(u, 129\cdot5K)$; so suppose for a contradiction that the latter holds.
    Since $B_G(B_1, \lceil(5/2)K\rceil)$ and $B_G(B_2, \lceil(5/2)K\rceil)$ are disjoint and not joined by an edge, the internal vertices of any $N_1$--$N_2$ path in $G'$ are in $C$; in particular, $\mathring{Q_1}, \mathring{Q_2} \subseteq C$. Moreover, $\dist_G(Q_1, Q_2) \geq 5K$ because $\dist_{G'}(Q_1, Q_2) \geq 5K$ and any $C$-path in $G$ (and hence any $\mathring{Q}_1$--$\mathring{Q}_2$ path) meeting $G - G'$ has length at least $2\cdot(\lceil(5/2)K\rceil+1) \geq 5K+2$. Since also $\dist_G(Q_1 \cup Q_2, P) \geq \dist_G(G[C,1], P) \geq 5K$ by assumption, the paths $Q_1, Q_2, P$ are pairwise at least $5K$ apart. Thus, as $\mathring{Q}_1$ and $\mathring{Q}_2$ lie in the same component $C$ of $G - (B_1 \cup B_2)$, applying \cref{lem:FindingAFatK4} yields that $G$ has a $K$-fat $K_4$ minor, which is a contradiction.

    It remains to show the `moreover'-part. For this, let $i \in [2]$ and let $C'$ be a component of $G- (B_i \cup B')$ which attaches to both $B_i$ and $B'$. If $B'$ does not meet $B_{3-i}$, we extend the path $P$ to a $B_i$--$B'$ path $P'$ by adding first a $P$--$C$ path $P_1$ through $B_{3-i}$ and then a $P_1$--$B'$ path $P_2$ in~$C$. Otherwise, we extend $P$ by a $P$--$B'$ path $P_1$ through $B_{3-i}$.
    We claim that $\dist_G(P', G[C',1]) \geq 5K$.
    For this, let $S$ be a shortest $G[C',1]$--$P'$ path in $G$. 
    We show that $S$ has length at least $5K$. Then $\dist_G(G[C',1], P') \geq ||S|| \geq 5K$, which yields the claim.
    Since $\dist_G(G[C',1], P) \geq \dist_G(G[C,1],P) \geq 5K$ as $C' \subseteq C$, we are done if $S$ ends in $P$. Hence, we may assume that $S$ ends in $P_1$ or $P_2$. 

    We distinguish two cases.
    First, assume that $S$ meets $B_G(w, 129 \cdot 5K) \subseteq B'$.
    Since $P_1 \cup P_2$ does not meet $B' = B_G(B_G(w, 129 \cdot 5K), 5K)$ except in its last vertex, we have $\dist_G(P_1 \cup P_2, B_G(w, 129 \cdot 5K)) \geq 5K$. Hence, as $S$ meets both $P_1 \cup P_2$ and $B_G(w, 129 \cdot 5K)$ by assumption, it has length at least~$5K$.

    Now assume that $S$ avoids $B_G(w, 129 \cdot 5K)$. Note first that also $C'$ and $P_1 \cup P_2$ avoid $B_G(w, 129 \cdot 5K)$. Moreover, $C', P_1 \cup P_2 \subseteq G[C,1] \cup G[B_1 \cup B_2]$. Since $B_G(w, 129 \cdot 5K)$ separates $B_1$ and $B_2$ in $G' \supseteq G[C,1] \cup G[B_1 \cup B_2]$, and $C' \subseteq C$ attaches to $B_i$ while $P_1 \cup P_2$ meets $B_{3-i}$, the two subgraphs $C'$ and $P_1 \cup P_2$ lie in distinct components of $(G[C,1] \cup G[B_1 \cup B_2]) - B_G(w, 129 \cdot 5K)$. Since $S$ avoids $B_G(w, 129\cdot5K)$ and $C$ is a component of $G - (B_1 \cup B_2)$, it follows that $S$ leaves $G[C,1] \cup G[B_1 \cup B_2]$ through $B_i$ and enters it again through $B_{3-i}$. Hence, $S$ contains a $B_1$--$B_2$ path. Since $\dist_G(v_1, v_2) \geq 2 r_1 + r_2 + 5K + 2$, we have $\dist_G(B_1,B_2) \geq 5K+2$, and thus~$S$ has length at least~$5K$.
\end{proof}

\begin{lemma}\label{lem:ComponentAttachingToTwoBalls}
    Let $G$ be a graph, and $K, r_1, r_2 \in \N_{\geq 1}$ such that $r_1 \leq r_2$. 
    Let $v_1, v_2$ be two vertices of $G$ that are at least $r_1 + r_2 + 5K + 2$ apart. Set $B_i := B_G(v_i, r_i)$ for $i \in [2]$, let $C$ be a component of $G - (B_1 \cup B_2)$ which attaches to $B_1$ and $B_2$, and suppose there exists a $B_1$--$B_2$ path $P$ in $G$ such that $\dist_G(P, G[C,1]) \geq 5K$.
    
    If $G$ has no $K$-fat $K_4$ minor, then there exists an induced subgraph $Y$ of $G[C,1]$ which admits an honest decomposition $(H, \cV)$ modelled on a graph $H \in \cH_{SP}$ with terminals $h_1, h_2$ such that
    \begin{enumerate}[label=\rm{(\alph*)}]
        \item \label{itm:CompAttToTwoBalls:FirstBag} $V_{h_i} = N_G(C) \cap B_i$ for $i \in [2]$,
        \item \label{itm:CompAttToTwoBalls:Rad1}
        $\rad_G(V_h) \leq R'_0(K)$ for all $h \in V(H)$ with $h \notin B_H(\{h_1, h_2\},1)$,
        \item \label{itm:CompAttToTwoBalls:Rad2}
        $\dist_G(v_i, v) \leq r_i + 2\cdot R_0(K) + 5K + 1$ for $i \in [2]$ and all $v \in V_{h}$ with $h \in B_H(h_i, 1)$,
        \item \label{itm:CompAttToTwoBalls:RadialSpread}
        $\irs(\cH) \leq 3$; moreover, $\dist_H(h_i, h) \leq 3$ for all $h \in V(H_v)$ and $v \in N_G(C) \cap B_i$ for $i \in [2]$, and
        \item \label{itm:CompAttToTwoBalls:CompsNhood} for every component $C'$ of $G-Y$ which meets, or equivalently is contained in, $C$, its neighbourhood $N_G(C')$ is contained in some bag $V_h$ of $(H, \cV)$.
    \end{enumerate}
\end{lemma}
 
We remark that the proof of \cref{lem:ComponentAttachingToTwoBalls} is the only point in this paper where it actually makes a difference that we not only consider finite but also infinite graphs. In fact, for finite graphs, a much simpler inductive argument yields the desired \gd\ $(H, \cV)$: 

\begin{proof}[Proof sketch of \cref{lem:ComponentAttachingToTwoBalls} for finite graphs]\let\qed\relax
    We may apply \cref{prop:BallMeetingAllPathsInComp} to $C$, $v_1, v_2$ and $K, r_1, r_2$, which yields a vertex $u \in B_G(C,1)$. 
    We now apply induction to the (now strictly smaller since $C$ is finite) components $D$ of $C - (B_1 \cup B')$ and $C-(B_2 \cup B')$ for $B' := B_G(u, 130\cdot 5K)$ that attach to both balls, to obtain \gd s $(H^D, \cV^D)$ with terminals $h_1, u$ or $u, h_2$ of $H^D$, respectively. We then obtain the desired \gd\ $(H, \cV)$ by gluing the $(H^D, \cV^D)$ together in $u$ and $h_1$ or $h_2$.
\end{proof}

\begin{proof}[Proof of \cref{lem:ComponentAttachingToTwoBalls} for arbitrary graphs]
    We define the desired \gd\ $(H, \cV)$ recursively (see \cref{fig:CompAttToTwoBalls} for a sketch). 
    Initialise $H^0$ as a $K_2$ on vertices $h_1 := v_1$ and $h_2 := v_2$. In particular, $H^0 \in \cH_{SP}$ with terminals $h_1$ and $h_2$. Set $V_{h_i} := N_G(C) \cap B_i$ and $r_{h_i} := r_i$ for $i \in [2]$, and note that $N_G(C) \subseteq V_{h_1} \cup V_{h_2}$. 
    Set further $\cC^0 := \{C\}$, $w_1^C := h_1$ and $w_2^C := h_2$.
    We warn the reader that in the following the vertices of the $H^i$ are vertices of $G$, i.e.\ $V(H^i) \subseteq V(G)$. Hence, we may divert from our usual convention, and denote the vertices of $H$ by $u,w$.
    
    Now suppose that for $n \in \N$ we have constructed graphs $H^0 \subseteq \dotsc \subseteq H^n \in \cH_{SP}$ with terminals~$h_1$ and~$h_2$, and sets $V_u \subseteq B_G(C,1)$ for all $u \in V(H^n)$.
    Assume further that, for every $i \leq n$, we fixed 
    \begin{enumerate}[label=\rm{(\Roman*)}]
        \item\label{itm:fixcomponents} the collection $\cC^i$ of all components $D$ of $C - \bigcup_{u \in H^i} V_u$ whose neighbourhood is not contained in a $V_u$ for some $u \in V(H^i)$,
        \item\label{itm:edgeforcomp} for every $D \in \cC^i$, an edge $w_1^{D}w_2^{D}$ of $H^i$,
        \item\label{itm:newvertexforcomp} for every $D \in \cC^{i-1} \setminus \cC^i$,  some $u_{D} \in V(H^i) \setminus V(H^{i-1})$ such that $w_1^Du_D, u_Dw_2^D \in E(H^i)$, if $i \geq 1$,
        \item\label{itm:newcompforvertex} for every $u \in V(H^i) \setminus V(H^{i-1})$, some $C_u \in \cC^{i-1}$  with $u \in B_G(C_u,1)$ and $u_{C_u} = u$, if $i \geq 1$, and
        \item \label{itm:Proof5} for all $u \in V(H^i) \setminus \{h_1, h_2\}$, the bag $V_u := B_G(u, R_0(K)) \cap B_G(C_u,1)$ and $r_u := R_0(K)$,
    \end{enumerate}  
    such that 
    \begin{enumerate}[label=\rm{(\roman*)}]
        \item \label{itm:Proof2} for every edge $e$ of $G[\bigcup_{u \in H^i} V_u]$ that is not an edge of $\bigcup_{u \in H^i} G[V_u]$ there is an edge $uu' \in E(H^i)$ such that $e \in G[V_{u} \cup V_{u'}]$;
        \item \label{itm:Proof3} every $D \in \cC^i$ is a component of $G - (V_{w^{D}_1} \cup V_{w^{D}_2})$;
        \item \label{itm:Proof4} for every $u \in V(H^i) \setminus V(H^{i-1})$ every $V_{w^{C_u}_1}$--$V_{w^{C_u}_2}$ path through $C_u$ intersects $B_G(u, 129\cdot 5K)$;
        \item \label{itm:Proof6} if $i \geq 1$, then for every $D \in \cC^{i-1}$ we have $D \notin \cC^i$ if and only if $\dist_G(w^{D}_1, w^{D}_2) \geq r_{w^{D}_1} + r_{w^{D}_2} + 5K + 2$;
        \item \label{itm:Proof3b} for every $D \in \cC^i \setminus \cC^{i-1}$ and the unique component $D' \in \cC^{i-1} \setminus \cC^{i}$ with $D \subseteq D'$  we have $w^{D}_1 = w^{D'}_1$, $w^{D}_2 = u_{D'}$ or $w^{D}_1 = u_{D'}$, $w^{D}_2 = w^{D'}_2$;
        \item \label{itm:Proof3c} for every $D \in \cC^i$ there exists a $B_G(w^{D}_1, r_{w^{D}_1})$--$B_G(w^{D}_2, r_{w^{D}_2})$ path $P$ with $\dist_G(P, G[D,1]) \geq 5K$.
    \end{enumerate}
    In the following, we will refer to components $D \in \cC^n$ as in \cref{itm:Proof6}, that is those which satisfy $\dist_G(w^{D}_1, w^{D}_2) \geq r_{w^{D}_1} + r_{w^{D}_2} + 5K + 2$, as \defn{long}.

    We now obtain $H^{n+1}$ as follows. 
    For every long component $D \in \cC^n$, let $u_{D}\in B_G(D,1)$ be a vertex obtained from the application of \cref{prop:BallMeetingAllPathsInComp} to $D$, $w_1^{D}, w_2^{D}$ and $r_{w^{D}_1}, r_{w^{D}_2}$, which is possible due to \cref{itm:Proof3c}. 
    Then $H^{n+1}$ is obtained from $H^n$ by adding for every long component $D \in \cC^n$ the vertex $u_{D}$ and the edges $w_1^{D}u_{D}$ and $u_{D}w_2^{D}$. 
    
    Let the collection $\cC^{n+1}$, the sets $V_{u_D}$ and the integer $r_{u_D}$ be as required by \cref{itm:fixcomponents} and \cref{itm:Proof5}, respectively.
    The construction obviously ensures \cref{itm:newvertexforcomp} and \cref{itm:Proof6}.
    For \cref{itm:newcompforvertex}, we set $C_{u_D} := D$ for every $u_D \in V(H^{n+1}) \setminus V(H^n)$, which also yields  \cref{itm:Proof4}.
    For \cref{itm:edgeforcomp}, let $D' \in \cC^{n+1} \setminus \cC^{n}$.
    Then $D'$ is contained in a long component $D \in \cC^n$.
    As the neighbourhood of $D'$ is not contained in a $V_u$ for $u \in V(H^{n+1})$, the ball $B_G(u_D,129 \cdot 5K)$ meets every $V_{w_1^{D}}$--$V_{w_2^D}$ path through $D$ and $B_G(u_D,129 \cdot 5K) \cap V(D) \subseteq V_{u_D}$, the component $D' \subseteq D$ attaches to $V_{u_D}$ and precisely one of $V_{w_1^D}$ and $V_{w_2^D}$, say $V_{w_1^D}$. 
    Then we set $w_1^{D'} := w_1^D$ and $w_2^{D'} := u_D$, which immediately ensures \cref{itm:Proof3} and \cref{itm:Proof3b}.
    The moreover-part of \cref{prop:BallMeetingAllPathsInComp} yields \cref{itm:Proof3c}.
    For \cref{itm:Proof2}, we first note that every edge of $G[\bigcup_{u \in H^{n+1}} V_u]$ that has no endvertex in $V_{u_D} \cap V(D)$ for any $D \in \cC^n$ was already an edge of $G[\bigcup_{u \in H^{n}} V_u]$, and hence satisfies \cref{itm:Proof2} by \cref{itm:Proof2} of $(H^n, \cV^n)$. Now let $e$ be any edge of $G[\bigcup_{u \in H^{n+1}} V_u]$ which is not in $G[\bigcup_{u \in H^{n}} V_u]$. Then there is $D \in \cC^n$ such that $e =xy$ has one endvertex, say $x$, in $V_{u_D} \cap V(D)$, so $e$ has its other endvertex $y$ in $G[D,1]$.
    Hence, by \cref{itm:Proof3} of $(H^n, \cV^n)$, the vertex $y$ is in $V_{w_1^D} \cup V_{w_2^D}$. By construction, $w_1^Du_D, u_Dw_2^D \in E(H^{n+1})$, and hence $e$ satisfies \cref{itm:Proof2}.
    Also $H^{n+1}$ is still in $\cH_{SP}$ with terminals $h_1, h_2$ by \cref{itm:edgeforcomp} and \cref{obs:ConstrOfSPGraphs}.
     This completes the verification that $H^{n+1}$ is as desired.

    \begin{figure}[ht]
        \centering
            \pdfOrNot{\def\svgwidth{0.65\columnwidth} 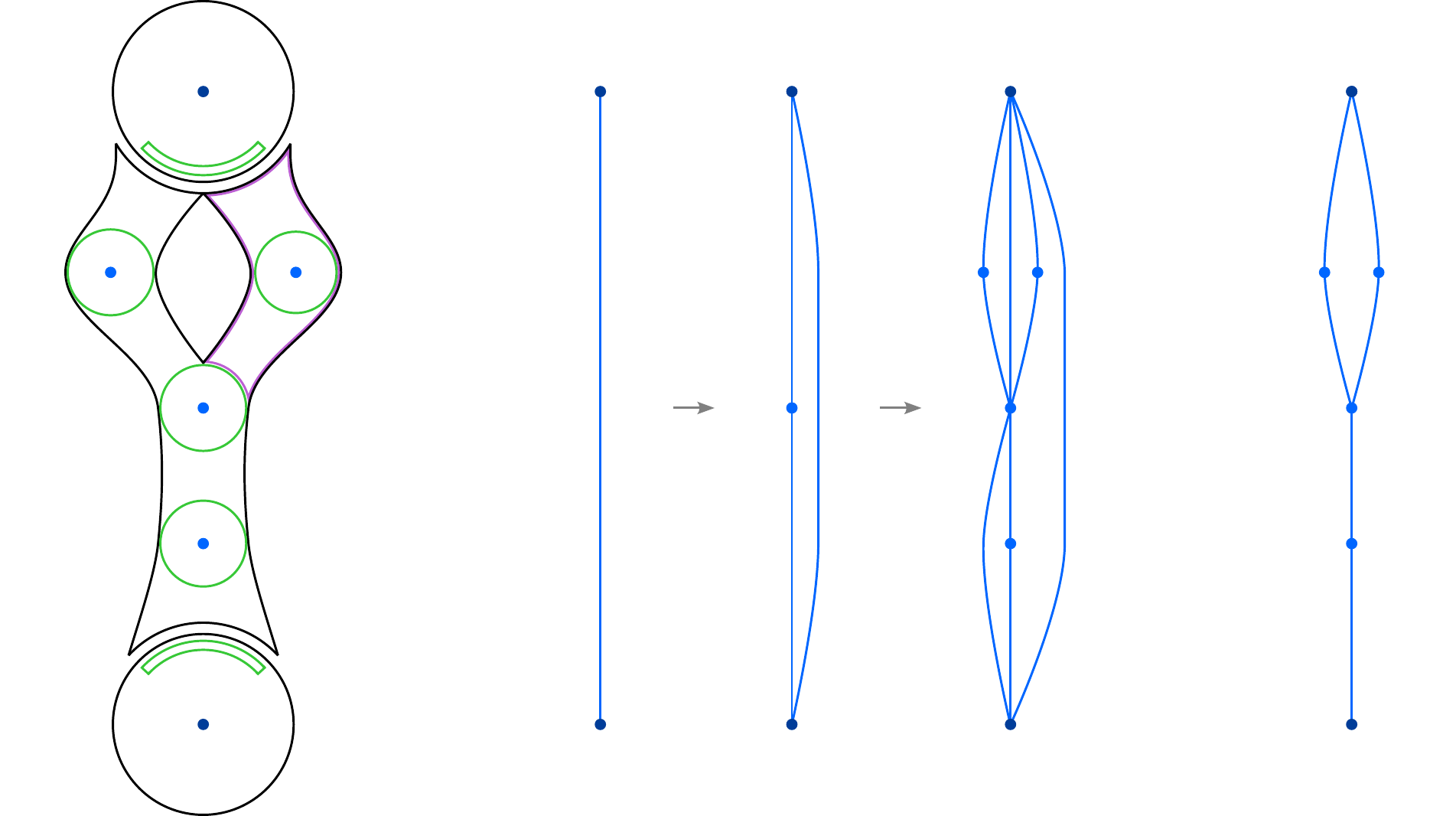}{\includesvg[width=0.65\columnwidth]{svg/limitseriesparallel.svg}}
        \caption{Depicted are the graphs $H^0, H^1, H^2, H$ in blue and the bags $V_u$ for the nodes $u$ of $H^2$ for the component $C$ of $G-(B_1 \cup B_2)$ in green. In this example, we have $w_1^{D} = v_1, w_2^{D} = u_1^1$ for the component $D = C_{u_2^2}$ in purple, and $H' = H^2$.}
        \label{fig:CompAttToTwoBalls}
    \end{figure}

    Let $H' := \bigcup_{i \in \N} H^i$ be the limit of the $H^i$. 
    By \cref{obs:ConstrOfSPGraphs}, we have $H' \in \cH_{SP}$ with terminals $h_1, h_2$.
    Now let $H$ be obtained from $H'$ by first deleting all edges $uu'$ of $H'$ with $\dist_G(u,u') \geq r_{u} + r_{u'} + 5K + 2$ and then subdividing each remaining edge precisely once. By \cref{obs:ConstrOfSPGraphs}, $H \in \cH_{SP}$ with terminals $h_1$ and $h_2$.
    We claim that $(H, \cV)$ with the $V_h$ already defined for all $u \in V(H')$ and $V_h := V_u \cup V_w$ for all new subdivision vertices $h$ on an edge $uw$ of $H'$ is a \gd\ of $Y := G[\bigcup_{u \in H} V_u] \subseteq G[C,1]$ as desired. 
 
    We first show that $(H, \cV)$ is a \gd\ of $Y$. Indeed, \cref{GraphDecomp:H1} holds, as the parts $G[V_u]$ cover $Y$ by the definition of $Y$ and the edges relevant for \cref{itm:Proof2} have not been deleted. Also $(H, \cV)$ satisfies \cref{GraphDecomp:H2}: 
    Let $v$ be in $V(Y)$, and let $n$ be minimal such that $v \in \bigcup_{h \in H^n} V_h$. Assume that $n \geq 1$. Then $v \in V_{h_v}$ for some $h_v \in V(H^n) \setminus V(H^{n-1})$. By \cref{itm:Proof5} and the minimality of $n$, the vertex $v$ lies in $C_{h_v} \in \cC^{n-1}$. Thus, the choice of $h_v$ is unique by \cref{itm:newcompforvertex} and \cref{itm:Proof5}.
    If $n= 0$, then we have $v \in N_G(C) \cap B_i$ for a unique $i \in [2]$ and $h_v = h_i$.
    We first prove that $H'_v = H'[\{h \in V(H') \mid v \in V_h\}]$ is contained in $H'[h_v,1]$.
    
    For this, let $h' \in V(H') \setminus \{ h_v \}$ with $v \in V_{h'}$ be given, and let $m \in \N$ be minimal such that $C_{h'} \in \cC^m$. 
    Since $n$ is minimal and $h_v$ is unique, we have that $m > n$ and that $C_{h'}$ attaches to $v$.
    For every $\ell$ with $m \geq \ell > n$, let $C^\ell$ be the unique component in $\cC^\ell$ which contains $C_{h'}$; in particular, $C_{h'} = C^m$.
    Since $C_{h'} \in \cC^m \setminus \cC^{m-1}$, it follows from \cref{itm:fixcomponents}, \cref{itm:Proof6} and \cref{itm:Proof5} that  no $C^\ell$ is in $\cC^{\ell-1}$.
    Also, every $C^\ell$ with $\ell > n$ attaches to $v$ as $v \in V_{h_v}$ and $C^m = C_{h'}$ attaches to $v$.
    By iteratively combining this with \cref{itm:Proof3} and \cref{itm:Proof3b}, we obtain that $h_v = w_2^{C^{n+1}}  = \dots = w_2^{C^m}$ (after possibly swapping the names of $w_1^{C^\ell}$ and $w_2^{C^\ell}$).
    Moreover, by \cref{itm:newcompforvertex}, $h' = u_{C_{h'}} = u_{C^{m}}$, and thus $h_vh' = w_2^{C^m}u_{C^m}$ is an edge of $H'$ by \cref{itm:newvertexforcomp}, which completes the proof that $H'_v \subseteq H'[h_v, 1]$. 
    
    Then \cref{GraphDecomp:H2} of $(H, \cV)$ follows:
    For every node $h'$ of $H'$ with $v \in V_{h'}$, since $\dist_G(h_v, h') \leq r_{h_v} + r_{h'}$ as $V_{h_v}$ and $V_{h'}$ intersect in $v$, we did not delete the edge $h_vh'$ in the construction of $H$. 
    Since $H$ is obtained from $H'$ by subdividing each remaining edge precisely once, and since we associated to every subdivision vertex $u$ of $H$ the union $V_u$ of the two bags associated with its neighbours, $H'_v \subseteq H'[h_v, 1]$ implies that $H_v = H[\{h \in V(H) \mid v \in V_h\}]$ is contained in $H[h_v, 3]$.\footnote{We remark that $H_v$ is not necessarily contained in $H[h_v, 2]$: for any given $h' \in H'[h_v, 1]$ and any edge $e \neq h'h_v$ incident with $h'$, the bag corresponding to the subdivision vertex of $e$ in $H$ also contains $v$ but has distance $3$ to $h_v$.} In particular, $H_v$ is connected, which yields \cref{GraphDecomp:H2}. Moreover, $\irs((H,\cV)) \leq 3$ and also \cref{itm:CompAttToTwoBalls:RadialSpread} follows immediately.

    We now check the other required properties.
    By the definition of $V_{h_1}$ and $V_{h_2}$ at the beginning, $(H, \cV)$ satisfies \cref{itm:CompAttToTwoBalls:FirstBag}.
    We claim that \cref{itm:CompAttToTwoBalls:Rad1} and \cref{itm:CompAttToTwoBalls:Rad2} hold.
    Indeed, for every $u \in V(H')$ we have $\dist_G(v,u) \leq r_u$ for all $v \in V_u$, where $r_u = R_0(K)$ for $u \neq h_1,h_2$ and $r_u = r_i$ for $u = h_i$.
    Now let $u \in V(H)$ be a subdivision vertex of an edge $ww'$ of $H'$, say with $r_w \geq r_{w'}$.
    Then $d_G(v,w) \leq r_w + 2 r_{w'} + 5K + 1$, which yields \cref{itm:CompAttToTwoBalls:Rad1} and \cref{itm:CompAttToTwoBalls:Rad2}.
    
    It thus remains to verify \cref{itm:CompAttToTwoBalls:CompsNhood}. For this, suppose towards a contradiction that there is a component $C'$ of $G-Y$ which meets $C$ and whose neighbourhood is not contained in some $V_h$. In particular, $C'$ is a component of $C-Y$. Now first suppose that there are no two vertices in $H'$ whose union of their bags contains $N_G(C')$. Then there are three vertices $u_1, u_2, u_3 \in N_G(C)$ and nodes $g_1, g_2, g_3$ such that $V_{g_i}$ contains $u_i$ but no other $u_j$. Pick some $n$ such that all $g_i$ are contained in $V(H^n)$. Then the component of $C - \bigcup_{h \in H^n} V_h$ that contains $C'$ contradicts that $H^n$ satisfied \cref{itm:Proof3}. Thus, there are nodes $g_1, g_2 \in V(H')$ such that $N_G(C') \subseteq V_{g_1} \cup V_{g_2}$. In particular, $C'$ is a component of $C - \bigcup_{h \in H^n} V_h$ for all $n \in \N$ with $g_1, g_2 \in V(H^n)$. Thus, $C' \in \cC^n$ for all such $n$, as its neighbourhood is not even contained in some $V_h$ with $h \in V(H')$. Let $m$ be the minimal such $n$. By \cref{itm:edgeforcomp} and \cref{itm:Proof3}, $C'$ is a component of $G-(V_g \cup V_{g'})$ and $gg'= w_1^{C'}w_2^{C'}$ is an edge of $H^n$. If we did not delete $gg'$ when constructing $H$, then we have a contradiction, as the bag $V_u$ associated with the subdivision vertex $u$ of $H$ on the edge $gg'$ then contains the entire neighbourhood of $C'$. Otherwise, we have $\dist_G(g, g') > r_{g} + r_{g'} + 5K + 2$, which by \cref{itm:Proof6} contradicts that $C' \in \cC^{m+1}$.
\end{proof}

\begin{proof}[Proof of \cref{lem:CompWithBipOfNhood}]
    Set $B_i := B_G(v_i, R_1(K))$ for $i \in [2]$. Further, let $P_i$ be a shortest $w$--$B_i$ path and set $W := P_1 \cup P_2$. Then $W$ is a $B_1$--$B_2$ walk, and hence $W$ contains a $B_1$--$B_2$ path $P$. By construction, we have $V(P) \subseteq V(W) \subseteq B_G(w, r-R_1(K)+1)$ and hence $\dist_G(P,G[C,1]) \geq R_1(K)-1 \geq 5K$. Since also $\dist_G(v_1, v_2) \geq 2R_1(K) + 5K + 1$ by assumption, we may apply \cref{lem:ComponentAttachingToTwoBalls} to every component $C'$ of $G-(B_1 \cup B_2)$ that has a neighbour in both $B_1$ and $B_2$ and that meets, or equivalently is contained in,~$C$. Since we are done if $G$ has a $K$-fat $K_4$ minor, we may assume that \cref{lem:ComponentAttachingToTwoBalls} yields for every such~$C'$ a partial \gd\  $\cH^{C'} = (H^{C'}, \cV^{C'})$ of $G$ with support $Y^{C'} \subseteq G[C',1] \subseteq G[C,1]$ that satisfies \cref{itm:CompAttToTwoBalls:FirstBag} to \cref{itm:CompAttToTwoBalls:CompsNhood}.

    We now define the desired decomposition $(H, \cV)$ of $Y := G[(\bigcup V(Y^{C'})) \cup ((B_1 \cup B_2) \cap V(C)) \cup N_G(C)]$. 
    The graph $H$ is obtained from the disjoint union of the $H^{C'}$ and three new vertices $h_1, h_2$ and $g$ by identifying all $h^{C'}_1$ with $h_1$ as well as all $h^{C'}_2$ with $h_2$ and adding the edges~$h_1g$ and~$h_2g$. Note that this is a parallel composition of graphs~$H^{C'} \in \cH_{SP}$ and a path of length $2$; so~$H$ is again in $\cH_{SP}$ with terminals $h_1$ and $h_2$ by \cref{obs:ConstrOfSPGraphs}; in particular, $H \in \Forbminor(K_4)$.
    Then $(H, \cV)$ with $V_{g} := N_G(C)$, $V_{h_i} := (B_i \cap V(C)) \cup (N_G(C) \cap B_G(B_i, 1))$ for $i \in [2]$ and $V_h := V^{C'}_h$ for all other $h \in V(H)$ where $C'$ is the unique component with $h \in H^{C'}$ is a decomposition of $Y$.
    $(H,\cV)$ obviously satisfies \cref{GraphDecomp:H1}, as the definition of the $V_{h_i}$ ensures that the $G[V_{h_i}]$ cover the edges from $\partial_G C$ to $N_G(C)$.
    Moreover, \cref{GraphDecomp:H2} holds for $(H,\cV)$, since the $V_{h_i}$ are disjoint from all components $C'$ of $G-(B_1 \cup B_2)$ which meet $C$. 

    Since $\partial_G C \subseteq B_1 \cup B_2$, we have $\partial_G C \subseteq V_{h_1} \cup V_{h_2} \subseteq V(Y)$.
    By construction, $(H, \cV)$ clearly satisfies \cref{itm:CompWithBipOfNhood:FirstBag} for $g \in V(H)$.
    Moreover, by \cref{itm:CompAttToTwoBalls:Rad1} and \cref{itm:CompAttToTwoBalls:Rad2} of the $\cH^{C'}$ and the definition of $V_{h_i}$, we have $\rad(V_h) \leq R_1(K) + 2R_0(K) + 5K + 2$ for all $h \neq g \in V(H)$, and thus $(H, \cV)$ satisfies \cref{itm:CompWithBipOfNhood:Radius}. 
    By construction, \cref{itm:CompWithBipOfNhood:Spread} follows from property \cref{itm:CompAttToTwoBalls:RadialSpread} of the $\cH^{C'}$. 
    
    Lastly, we show that $(H, \cV)$ satisfies \cref{itm:CompWithBipOfNhood:CompsNhood}. 
    Every component of $G-Y$ which does not meet $C$ is contained in $N_G(C) = V_g$. Now let $D$ be an arbitrary component of $G-Y$ which meets, or equivalently is contained in, $C$. Let $C'$ be the component of $G-(B_1 \cup B_2)$ which contains $D$. If $C'$ attaches only to $B_1$ or only to $B_2$, then $N_G(D)$ is contained in $V_{h_1}$ or $V_{h_2}$, respectively. So we may assume that $C'$ attaches to both $B_1$ and $B_2$. Then $N_G(D)$ is contained in $V^{C'}_h$ for some $h \in V(H^D)$ by \cref{itm:CompAttToTwoBalls:CompsNhood}.
    Thus, by construction of $(H, \cV)$ we have that $N_G(D) \subseteq V_h$ for some $h \in V(H) \setminus \{g\}$, and hence $\rad_G(V_h) \leq R_1(K) + 2R_0(K) + 5K + 2$, as we have shown for \cref{itm:CompWithBipOfNhood:Radius} above.
    Hence, $(H, \cV)$ is as desired.
\end{proof}

\subsection{Proof of \texorpdfstring{\cref{lem:CompWithThreeVertices}}{Lemma 5.3}} \label{subsec:Lemma2}

Let us briefly sketch the proof of \cref{lem:CompWithThreeVertices}. For this, let us first recall its premises: Let $G$ be a graph with no $K$-fat $K_4$ minor for $K \in \N_{\geq 1}$, let $B$ be some ball in $G$ around a vertex $w$ of radius $r \in \N$, and let $C$ be a component of $G-B$. 
Suppose that $C^*$ is a component of $C-B_G(B,22K+1)$ such that there are three vertices in $\partial_G C^*$ that are pairwise far apart in $G$. We first find a vertex $w' \in V(C^*)$ such that there are at least three $B_G(w', 22K)$--$B_G(w, r - \ell(K))$ paths that are pairwise at least $11K$ apart (\cref{lem:BallAndThreeComps}). These paths then have to lie in distinct components of $G' := G - (B_1(w', 22K) \cup B_G(w, r - \ell(K)))$ by \cref{lem:FindingAFatK4}. In particular, we can conclude that, for every component of $G'$ that attaches to both balls, one of those three paths is still at least $5K$ away. We then apply \cref{lem:ComponentAttachingToTwoBalls} to $B_1 := B_G(w',22K)$ and $B'_2 := B_G(w, r-\ell(K))$ to obtain, for every component $D$ of $G'$ that attaches to $B_1$ and $B'_2$, a partial decomposition modelled on a graph $H^{D} \in \cH_{SP}$ with terminals $h_1^{D}, h_2^{D}$. We modify these decompositions by contracting all nodes of $H^D$ whose bags contain vertices from $B_G(B,1)$ to a single vertex $h_2$, which ensures that we may enlarge the bag associated with $h_2$ so that it contains $B$ (\cref{lem:CompAttToTwoBalls2}). We then glue these modified decompositions together by identifying the $h_1^{D}$ and the $h_2^{D}$. In the end, we adopt the decomposition so that it is $R$-component-feasible and obtain the desired partial decomposition of $G[C,1]$ as its restriction to $G[C,1]$.

We first show that there exists a vertex $w'$ as described above.

\begin{lemma}\label{lem:ThreePathsFarApart}
    Let $\ell, d, r \in \N_{\geq 1}$ with $\ell \geq 6d$, and let $B$ be a ball in a graph $G$ around a vertex $w$ of radius~$r$. Assume that $C$ is a component of $G-B$ whose boundary $\partial_G C$ contains three vertices that are pairwise at least $4\cdot(2\ell + d + 2)$ apart. Then there exist a vertex $w' \in V(C)$ and three $B_G(w,r-\ell)$--$B_G(w', 2d)$ paths that are pairwise at least $d$ apart.
\end{lemma}

\begin{figure}
    \centering
    \pdfOrNot{\def\svgwidth{0.4\columnwidth} 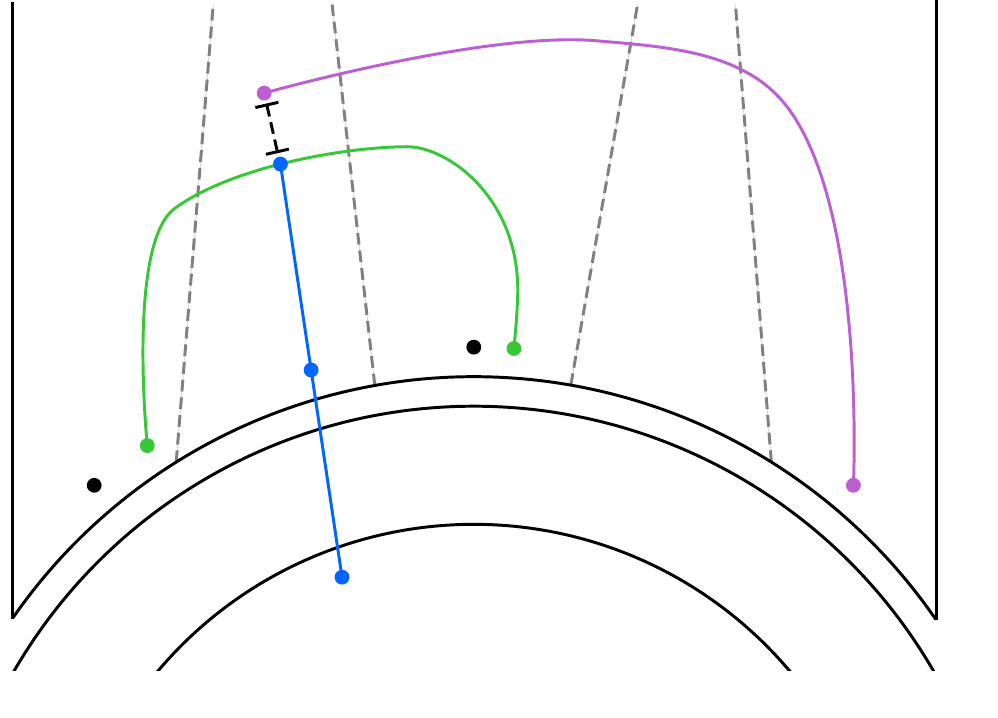}
    {\includesvg[width=0.4\columnwidth]{svg/setupthreenbs.svg}}
    \caption{Setup in the beginning of the proof of \cref{lem:ThreePathsFarApart}}
    \label{fig:ThreePathsFarApart}
\end{figure}

\begin{proof}
    Let $u_1, u_2, u_3 \in \partial_G C$ be three vertices that are pairwise at least $4\cdot(2\ell + d + 2)$ apart. Let $U_i$, for $i \in [3]$, be the set of vertices $v \in V(C)$ such that every shortest $v$--$B$ path meets $\partial_G C$ in a vertex of distance less than $2\ell + d + 2$ to $u_i$. Note that the sets $U_i$ are pairwise disjoint, as a shortest $v$--$B$ path meets $\partial_G C$ in precisely one vertex and this vertex has distance less than $2\ell + d + 2$ to at most one $u_i$, since the pairwise distance of the $u_i$ is at least $2\cdot(2\ell + d + 2)$.
    
    Let $P' = p'_0 \dots p'_n$ be a $U_1$--$(U_2 \cup U_3)$ path in $C$; by symmetry, we may assume that $P'$ ends in $U_2$. We extend $P'$ to a $(\partial_G C \cap U_1)$--$(\partial_G C \cap U_2)$ path $P = p_0 \dots p_m$ by adding a shortest $(\partial_G C \cap U_1)$--$p'_0$ path and a shortest $p'_n$--$(\partial_G C \cap U_2)$ path. Note that these paths are contained in $U_1$ and $U_2$, respectively, and hence are disjoint. Thus $P$ is indeed a path. 
    Further, let $Q$ be a $u_3$--$B_G(P, 2d)$ path in $C$ and let $p$ be some vertex of $P$ such that $Q$ ends in $B_G(p, 2d)$. Then $p \not\in U_3$ by the assumption on $P$. Let $W = w_0 \dots w_k$ be a shortest $p$--$B_G(w, r-\ell)$ path in $G$, and let $p'$ be the unique vertex in $W \cap \partial_G C$ (see \cref{fig:ThreePathsFarApart}). 
    Since $p \notin U_3$, we may choose $W$ so that $\dist_G(p', u_3) \geq 2\ell + d+2$. Moreover, as $\dist_G(u_1, u_2) \geq 4\cdot(2\ell + d+2)$, we cannot have $\dist_G(p', u_1), \dist_G(p', u_2) < 2\cdot(2\ell + d+2)$; by symmetry, we may thus assume $\dist_G(p', u_2) \geq 2\cdot(2\ell + d+2)$. In particular, this implies that 
    \begin{equation} \label{eq:ThreePathsFarApart:1}
    \dist_G(p_m, p') \geq \dist_G(p',u_2) - \dist_G(u_2, p_m) \geq 2\cdot(2\ell + d +2) - (2\ell + d+2) = 2\ell + d+2.
    \end{equation}
    
    Let $x_3$ be the first vertex on $Q$ such that $x_3 \in B_G(W, 2d)$, and let $x_2$ be the last vertex on $P$ such that $x_2 \in B_G(W, 2d)$. Further, let $i_j \in \{0, \dots, k\}$, for $j \in \{2,3\}$ be an index such that $x_j \in B_G(w_{i_j}, 2d)$. As the following reasoning will be symmetric in $P$ and $Q$, we may assume without loss of generality that $i_3 \leq i_2$. We claim that $w' := w_{i_2}$ is as desired (see \cref{fig:ThreePaths:2a}). 

    \begin{figure}
        \centering
        \begin{subfigure}[b]{0.45\linewidth}
                \pdfOrNot{\def\svgwidth{\columnwidth} 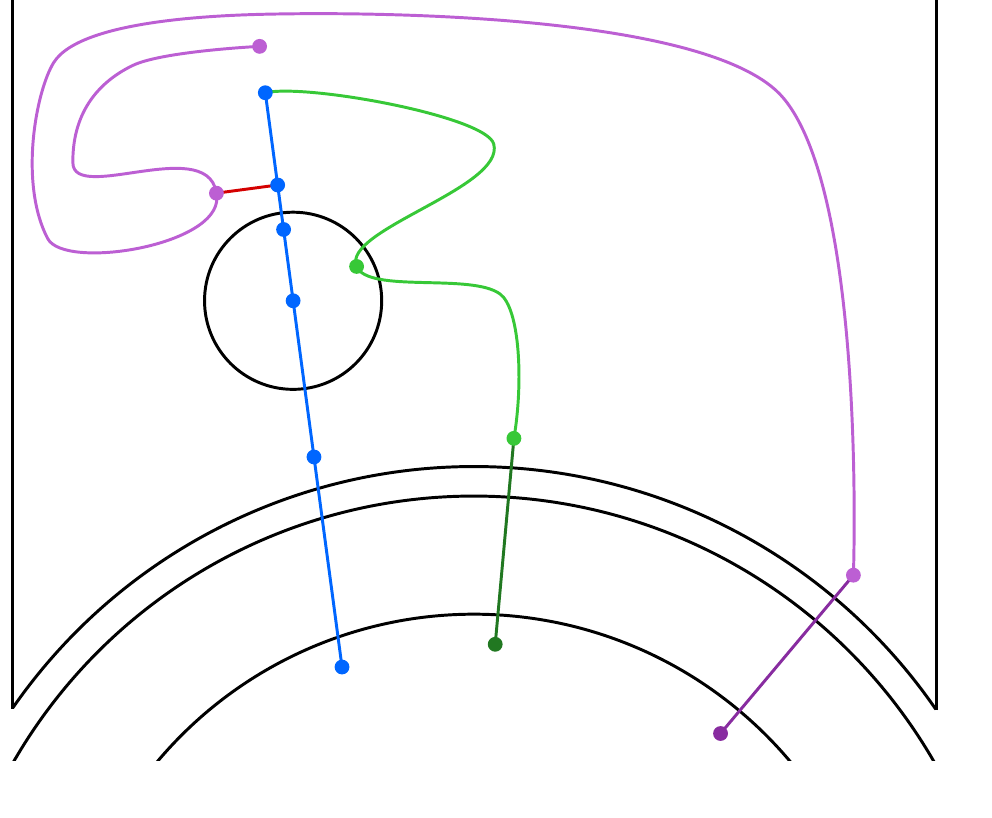}{\includesvg[width=\columnwidth]{svg/threenbs.svg}}
            \caption{}
            \label{fig:ThreePaths:2a}
        \end{subfigure}
        \hspace{10mm}
        \begin{subfigure}[b]{0.45\linewidth}
                \pdfOrNot{\def\svgwidth{\columnwidth} 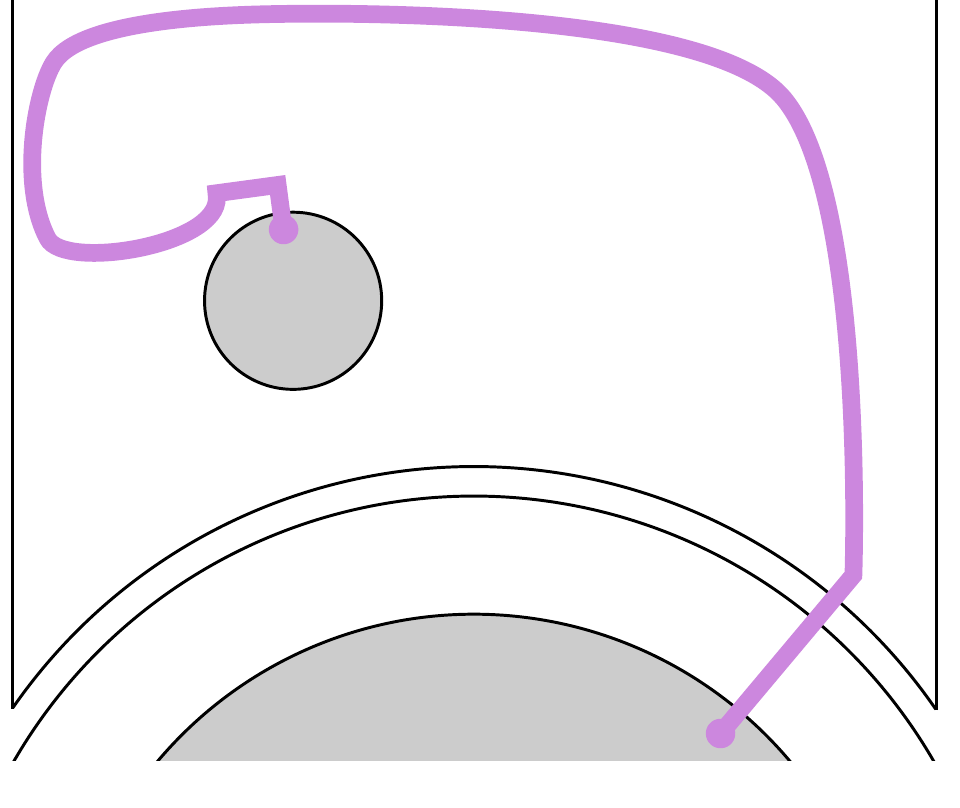}{\includesvg[width=\columnwidth]{svg/threenbs2.svg}}
            \caption{}
            \label{fig:ThreePaths:2b}
        \end{subfigure}
        \caption{Choice of $w'$ and construction of the paths $P_1, P_2, P_3$ in the proof of \cref{lem:ThreePathsFarApart}.}
        \label{fig:ThreePaths:2}
    \end{figure}

    For this, let $Q_2$ be a shortest $p_m$--$B_G(w, r-\ell)$ path and $Q_3$ a shortest $u_3$--$B_G(w, r-\ell)$ path. Further, let $S_3$ be a shortest $x_3$--$w_{i_3}$ path and let $s_3$ be the first vertex on the path $S'_3 := x_3S_3w_{i_3}Ww_{i_2}$ that is contained in $B_G(w_{i_2}, 2d)$ (see \cref{fig:ThreePaths:2a}). We now define our final three $B_G(w,r-\ell)$--$B_G(w', 2d)$ paths $P_1,P_2,P_3$ (see \cref{fig:ThreePaths:2b}): 
    \[
    P_1 := w_{i_2 + 2d}Ww_k,\; P_2 := x_2Pp_mQ_2 \text{ and } P_3 := s_3\bar{S}'_3x_3\bar{Q}u_3Q_3.
    \]
     We claim that the $P_i$ are as desired, i.e.\ $\dist_G(P_i, P_j) \geq d$ for every two distinct $i,j \in [3]$. First, we show $\dist_G(P_1, P_2) \geq d$: By the choice of $x_2$, we have $\dist_G(W, x_2Pp_m) \geq 2d \geq d$. Further,  since $Q_2$ starts in $p_m$ and both shortest $\partial_G C$--$B_G(w,r-\ell)$ paths $Q_2$ and $p'Ww_k$ have length $\ell + 1$, we have
    \begin{equation} \label{eq:ThreePathsFarApart:2}
    \dist_G(p'Ww_k, Q_2) \geq \dist_G(p_m, p') - ||p'Ww_k|| - ||Q_2|| \geq (2\ell + d + 2) - 2(\ell + 1) = d.
    \end{equation}
    Since $\dist_G(p_m,p') \geq 2\ell+d$ by \eqref{eq:ThreePathsFarApart:1}, $\ell \geq d$, $Q_2$ is a shortest $\partial_G C$--$B_G(w, r-\ell)$ path and $p' \in \partial_G C$, we have $\dist_G(Q_2, B_G(p',d) \cap V(W)) \geq d$.
    As $W$ is a shortest $p$--$B_G(w,r-\ell)$ path, the remainder of $W$ in $C$ has distance $\geq d$ to $B$ and hence to $Q_2$.
    All in all, $\dist_G(W, Q_2) \geq d$, and hence $\dist_G(P_1, P_2) \geq d$.

    Next, we show $\dist_G(P_1, P_3)$:
    Analogously to $\dist_G(W, Q_2) \geq d$, we also find $\dist_G(W, Q_3) \geq d$. 
    Moreover, $\dist_G(W, x_3\bar{Q}u_3) \geq 2d \geq d$ by the choice of $x_3$.
    Thus, to see that $\dist_G(P_1, P_3) \geq d$, it suffices to verify $\dist_G(w_{i_2+2d}Ww_k, S'_3) \geq d$: Since $W$ is a shortest path in $G$ and $i_3 \leq i_2$, we have $\dist_G(w_{i_2+2d}Ww_k, w_{i_3}Ww_{i_2}) = \dist_G(w_{i_2+2d}, w_{i_2}) = 2d$. Moreover, by the choice of $x_3 \in V(Q)$ and $S_3$, the $x_3$--$w_{i_3}$ path $S_3$ has length $||S_3|| = \dist_G(W,x_3) = 2d$. Thus, the distance from the first half of $S_3$ to $w_{i_2+2d}Ww_k$ is $\geq \dist_G(x_3,W) - ||S_3||/2 = 2d - d = d$, and the distance from the second half of $S_3$ to $w_{i_2+2d}Ww_k$ is $\geq \dist_G(w_{i_3},w_{i_2+2d}Ww_k) - ||S_3||/2 \geq 2d - d = d$.

    Finally, we show $\dist_G(P_2, P_3) \geq d$: 
    By the choice of $Q$, we have $\dist_G(P, Q) \geq 2d$.
    Similarly to the previous argument, we find $\dist_G(x_2Pp_m, S'_3) \geq d$, as $\dist_G(W, x_2Pp_m), \dist_G(P,Q) \geq 2d$, and thus the distance from the first half of $S_3$ to $x_2Pp_m$ is $\geq \dist_G(x_3,P) - ||S_3||/2 \geq 2d - d = d$, and the distance from the second half of $S_3$ to $x_2Pp_m$ is $\geq \dist_G(w_{i_3}, x_2Pp_m) - ||S_3||/2 \geq 2d - d = d$.
    Moreover, analogously to \eqref{eq:ThreePathsFarApart:2}, we find 
    \[
    \dist_G(Q_2, Q_3) \geq \dist_G(p_m, u_3) - ||Q_2|| - ||Q_3|| \geq \dist_G(U_2 \cap \partial_G C, u_3) - 2(\ell+1) \geq (2\ell+d+2) - 2(\ell+1) = d. 
     \]
    Further, $\dist_G(x_2Pp_m, Q_3), \dist_G(u_3Qx_3, Q_2) \geq d$ since $\dist_G(P, u_3), \dist_G(Q, p_m) \geq \dist_G(P,Q) \geq 2d$ and $P, Q \subseteq C \subseteq G-B$, while $Q_2, Q_3$ are shortest $\partial_G C$--$B_G(w, r-\ell)$ paths starting in $p_m, u_3$, respectively.
    It thus remains to show $\dist_G(S_3, Q_2) \geq d$: 
    For this, it suffices to prove $\dist_G(x_3,Q_2) \geq 3d$, as $S_3$ has length $2d$ and starts in $x_3$.
    If $\dist_G(x_3, B) > 3d$, then $\dist_G(x_3,Q_2) \geq 3d$ follows immediately, as $Q_2 \subseteq G[B,1]$.
    Thus, we may assume $\dist_G(x_3, B) \leq 3d$.
    Suppose for a contradiction that $\dist_G(x_3, Q_2) < 3d$.
    Then $\dist_G(x_3,p_m) \leq 2 \dist_G(x_3, Q_2) < 6d$, since $x_3 \in V(C)$ and $Q_2$ is a shortest $p_m$--$B_G(w, r- \ell)$ path.
    Also $\dist_G(w_{i_3}, p') \leq \dist_G(w_{i_3}, x_3) + \dist_G(x_3, B) \leq 5d$, since $W$ is a shortest path.
    Thus, $\dist_G(p_m, p') \leq \dist_G(p_m, x_3) + \dist_G(x_3, w_{i_3}) + \dist_G(w_{i_3}, p') < 13d$, which contradicts $\dist_G(p_m,p') \geq 2 \ell + d$ as $\ell \geq 6d$.  
\end{proof}

\begin{corollary}\label{lem:BallAndThreeComps}
    Let $K, \ell, r \in \N_{\geq 1}$ with $\ell \geq 66K$, and let $B$ be a ball in a graph $G$ around a vertex $w$ of radius~$r$. Suppose there is a component $C$ of $G-B$ such that $\partial_G C$ contains three vertices that are pairwise at least $4\cdot(2\ell+11K+2)$ far apart. 
    
    If $G$ has no $K$-fat $K_4$ minor, then there exists a vertex $w' \in V(C)$ such that for every component $D$ of $G - (B' \cup A)$ that attaches to $B' := B_G(w, r-\ell)$ and $A := B_G(w', 22K)$ there exists a $B'$--$A$ path $P$ with $\dist_G(P, G[D,1]) \geq 5K$.
\end{corollary}

\begin{proof}
    We note that in particular $r \geq \ell$ as $\partial_G C$ contains two vertices which are $\geq 4\cdot(2\ell+11K+2)$ apart.
    By \cref{lem:ThreePathsFarApart} (with $d := 11K$) there exists a vertex $w' \in V(C)$ and three $B'$--$B_G(w', 22K)$ paths $P_1, P_2, P_3$ such that $\dist_G(P_i, P_j) \geq 11K$ for all $i \neq j \in [3]$.
    Set $A := B_G(w', 22K)$, and let $D$ be a component of $G - (B' \cup A)$ that attaches to $B'$ and $A$.
    
    We show that there is some $i \in [3]$ such that $\dist_G(P_i, G[D,1]) \geq 5K$. 
    Towards a contradiction, suppose that $\dist_G(P_i, D) \leq 5K$ for all $i \in [3]$.
    It then follows that all three paths $P_i$ lie in the unique component $D'$ of $G - (B_G(w, r - \ell - 3K) \cup B_G(w', 19K))$ that contains $D$. For each $i \in [3]$, let $P'_i$ be a shortest $B_G(w, r-\ell-3K)$--$B_G(w', 19K)$ path in $P_i \cup G[B'] \cup G[A]$. Then $\dist_G(P'_i, P'_j) \geq 11K - 2 \cdot 3K = 5K$, which by \cref{lem:FindingAFatK4} ($\dist_G(w,w') > r_1 + r_2 + 5K +1$ with $r_1 = r-\ell-3K, r_2 = 19K$, since $\dist_G(w',w) \geq r$ and $\ell \geq 23K$) implies that $G$ has a $K$-fat $K_4$ minor, which contradicts our assumption.
\end{proof}

\begin{lemma}\label{lem:CompAttToTwoBalls2}
    Let $K, r_1, r_2 \in \N_{\geq 1}$ with $r_1 \leq R_0(K)$ and $r_2 > \ell(K)$, and let $v_1, v_2$ be two vertices of a graph $G$ which are at least $r_1 + r_2 + 2$ apart. Set $B_i := B_G(v_i, r_i)$ for $i \in [2]$ and $B'_2 := B_G(v_2, r_2-\ell(K)-1)$. Let $C$ be the component of $G - B_2$ that contains $v_1$. Let $D$ be a component of $G - (B_1 \cup B'_2)$ which attaches to $B_1$ and $B'_2$, and suppose there is a $B_1$--$B'_2$ path $P$ in $G$ such that $\dist_G(P, G[D,1]) \geq 5K$.
  
    If $G$ has no $K$-fat $K_4$ minor, then there exists an induced subgraph $Y$ of $G[B_1 \cup V(D) \cup B_2]$ which admits an honest decomposition $(H, \cV)$ modelled on some graph $H \in \cH_{SP}$ with terminals $h_1, h_2$ such that
    \begin{enumerate}[label=\rm{(\greek*)}]
        \item \label{itm:CompAttToTwoBalls2:FirstBag} $V_{h_1} = B_1$ and $V_{h_2} \supseteq B_2$,
        \item \label{itm:CompAttToTwoBalls2:Rad1} $\rad_G(V_h) \leq R_0'(K)$ for all nodes $h \neq h_2$ of $H$,
        \item \label{itm:CompAttToTwoBalls2:RadiusWidth} $\rad_G(V_{h_2}) \leq r_2+2R'_0(K)+1$,
        \item \label{itm:CompAttToTwoBalls2:RadiusSpread} $\irs((H, \cV)) \leq 3$, and
       \item \label{itm:CompAttToTwoBalls2:CompsNhood} for every component $D'$ of $G - Y$ that meets $C$ and $D$, we have $\rad_G(N_G(D')) \leq R_2(K) - 1$ and there is a node $h$ of $H$ such that $N_G(D') \subseteq V_h$.
    \end{enumerate}
\end{lemma}

\begin{proof}
    Since $\ell(K) \geq R_0(K) + 5K \geq r_1 + 5K$, applying \cref{lem:ComponentAttachingToTwoBalls} to $v_1, v_2, r_1, r'_2 := r_2-\ell(K)-1, K$ and $D$ yields an induced subgraph $Y^1 \subseteq G[D,1]$ and an honest decomposition $(H^1, \cV^1)$ of $Y^1$ modelled on a graph $H^1 \in \cH_{SP}$ with terminals $h_1, h_2$ which satisfies properties \ref{itm:CompAttToTwoBalls:FirstBag} -- \cref{itm:CompAttToTwoBalls:CompsNhood} from \cref{lem:ComponentAttachingToTwoBalls}. 
    
    Set $B_2^{+1} := B_G(v_2, r_2 + 1)$ and
    \[
    \tilde{H} := H^1[\{h \in V(H^1) \mid V^1_h \cap B_2^{+1} \neq \emptyset\}].
    \]
    \noindent We remark that $h_2$ and $N_{H^1}(h_2)$ are contained in $\tilde H$ since $(H^1, \cV^1)$ is honest and \ref{itm:CompAttToTwoBalls:FirstBag}, while $h_1 \notin V(\tilde{H})$ by \cref{itm:CompAttToTwoBalls:FirstBag} and because $v_1$ and $v_2$ are at least $r_2+r_1+2$ apart.
    
    We claim that $\tilde{H}$ is connected. 
    Indeed, since $h_2 \in V(\tilde{H})$, it is enough to find for every $h \in V(\tilde{H})$ some $h$--$h_2$ path in $\tilde{H}$. 
    Let $h \in V(\tilde{H}$) be given, pick some $v \in V^1_h \cap B_2^{+1} \subseteq B_G(D,1)$, and let $Q$ be some $v$--$(N_G(D) \cap B'_2)$ path in $G[B_2^{+1} \cap B_G(D,1)]$, which exists since $B_1$ and $B_2^{+1}$ are disjoint and $D$ is a component of $G-(B_1 \cup B_2')$. Set $Q^1 := Q \cap Y^1$, and note that $h_2$ is a node of $H_{Q^1} := H[\{h \in V(H^1) \mid V^1_h \cap Q^1 \neq \emptyset\}]$ since $V^1_{h_2} = N_G(D) \cap B'_2$ by \ref{itm:CompAttToTwoBalls:FirstBag}. Moreover, $H_{Q^1} \subseteq \tilde{H}$ by definition of $Q$. By \cref{GraphDecomp:H2'}, $H_{\tilde{Q}}$ is connected for every component $\tilde{Q}$ of $Q^1$, and thus the claim follows if $Q^1$ is connected. Otherwise, the claim follows as the neighbourhood of every component of $D - Y^1$ is contained in some bag of $(H^1, \cV^1)$ by \ref{itm:CompAttToTwoBalls:CompsNhood}, and thus $H_{\tilde{Q}}$ and $H_{\hat{Q}}$ intersect for `consecutive' components~$\tilde{Q}$ and~$\hat{Q}$ of $Q^1$, as $Q \subseteq B_G(D,1)$.

    We thus obtain a decomposition $(H, \cV^2)$ of $Y^1$ where $H := H^1/\tilde{H}$ is the graph obtained from contracting $\tilde{H}$ down to a single vertex, which we again call $h_2$, by merging all bags of $(H^1, \cV^1)$ to $V^1_{h_2}$ that contain a vertex of $B_2^{+1}$, i.e.\  $V^2_{h_2} = \bigcup_{h \in \tilde H} V_h^1$ and $V^2_h = V^1_h$ for every node $h \neq h_2$ of $H^2$.  Note that $H$ is still in $\cH_{SP}$ with terminals $h_1,h_2$.
    We now obtain the desired graph-decomposition $(H, \cV)$ by letting $V_h := V^2_h = V^1_h$ for all $h \in V(H) \setminus \{h_1, h_2\}$ and
    $V_{h_i} := V^2_{h_i} \cup B_i$ for $i \in [2]$.
    
    We claim that $(H, \cV)$ is a \gd\ of $Y := G[B_1 \cup V(Y^1) \cup B_2] \subseteq G[B_1 \cup V(D) \cup B_2]$. Indeed, $(H, \cV)$ satisfies \cref{GraphDecomp:H2} as $(H, \cV^2)$ already satisfied \cref{GraphDecomp:H2}, because, by construction, $V^2_{h_2}$ is the only bag of $(H, \cV^2)$ that may contain vertices of $B_2$, and because $B_1 \setminus V^1_{h_1}$ does not meet $Y^2$ and $B_2$ as $Y^1 \subseteq G[D,1]$ and because $B_1$ and $B_2$ are disjoint.
    
    To see that \cref{GraphDecomp:H1} holds, observe that, since $(H, \cV^2)$ is a \gd\ of $Y^1$, and $V_{h_2} \supseteq B_2$ and $V_{h_1} = B_1$, it suffices to show that the parts of $(H, \cV)$ cover all edges between $Y^1 \cap D$ and $B_1 \cup B_2$ as well as all edges between $B_1$ and $B_2$. 
    Since $N_G(D) \cap B_1 \subseteq V^2_{h_1}$ and $(H, \cV^2)$ satisfies \cref{GraphDecomp:H1}, the parts of $(H, \cV^2)$ already covered all edges between $B_1$ and $Y^1 \cap D$, and so the parts of $(H, \cV)$ do so as well.
    Further, the parts of $(H, \cV)$ cover all edges between $B_2$ and $Y^1 \cap D$ since $V(Y^1) \cap N_G(B_2) \subseteq V(Y^1) \cap B_2^{+1} \subseteq V_{h_2}$ as ensured by the contraction of $\tilde H$ in the construction of $(H^2, \cV^2)$.
    Finally, there are no edges between $B_1$ and $B_2$ because $\dist_G(v_1, v_2) \geq r_1 + 2 + r_2$.

    We now check that $(H, \cV)$ satisfies \cref{itm:CompAttToTwoBalls2:FirstBag} -- \cref{itm:CompAttToTwoBalls2:CompsNhood}.
    By definition, $V_{h_2} = V_{h_2}^2 \cup B_2 \supseteq B_2$ and by \cref{itm:CompAttToTwoBalls:FirstBag} $V_{h_1} = V_{h_1}^2 \cup B_1 = B_1$, and hence $(H, \cV)$ satisfies \ref{itm:CompAttToTwoBalls2:FirstBag}. 
    To see that $(H,\cV)$ satisfies \ref{itm:CompAttToTwoBalls2:Rad1}, let us first note that  $\rad(V_{h_1}) =\rad_G(B_1) \leq r_1 \leq R_0(K)$.
    
    So by \ref{itm:CompAttToTwoBalls:Rad1} and \ref{itm:CompAttToTwoBalls:Rad2}, we have $\rad_G(V_h) = \rad_G(V^1_h) \leq \max\{r_1 + 2R_0(K)+5K+1, R'_0(K)\} \leq R'_0(K)$ for all nodes $h \neq h_1$ of $V(H^1)$ with $h \notin B_{H^1}(h_2,1)$.
    So since $\tilde H$ contains $N_{H^1}(h_2)$ and $\tilde H$ was contracted to $h_2$ in the construction of $H$, this yields \ref{itm:CompAttToTwoBalls2:Rad1}.
    
    Let us now verify that $(H, \cV)$ satisfies \cref{itm:CompAttToTwoBalls2:RadiusWidth}. 
    For this, let $v \in V_{h_2}$. If $v \in B_2$, then $\dist_G(v, v_2) \leq r_2 \leq r_2 + 2 R_0'(K)+1$ as desired. Otherwise, by construction, $v$ is contained in $V^1_h$ for some node $h$ of $\tilde H \subseteq H^1$. 
    If $h \in B_{H^1}(h_2,1)$, then $\dist_G(v, v_2) \leq r'_2 + 2R_0(K)+5K+1$ by \ref{itm:CompAttToTwoBalls:Rad2}. Otherwise, by \ref{itm:CompAttToTwoBalls:Rad1}, we have $\dist_G(v_2, v) \leq r_2 + 1 + 2\rad_G(V^1_h) \leq r_2 + 1 + 2R'_0(K)$ because $V^1_h$ meets $B_2^{+1} = B_G(B_2, 1)$ as $h \in V(\tilde H)$. Thus, $(H, \cV)$ satisfies \cref{itm:CompAttToTwoBalls2:RadiusWidth}. Further, \cref{itm:CompAttToTwoBalls2:RadiusSpread} holds for $(H,\cV)$ by \cref{itm:CompAttToTwoBalls:RadialSpread}.
    
    We are thus left to check \ref{itm:CompAttToTwoBalls2:CompsNhood}. For this, let $D'$ be a component of $G - Y$ that meets $C$ and $D$; in particular, since $N_G(D) \cap B_1 \subseteq V(Y^1)$, we have that $D'$ is a component of $G - (V(Y^1) \cup B_2)$. Let $\hat{D}$ be the component of $D - Y^1$ that contains $D'$. Since $D' \cap C \neq \emptyset$, there is a $\hat{D}$--$B_1$ path $Q$ in $C$; let $q$ be its first vertex in $Y^1$; in particular, $q \in N_G(\hat D)$. Then 
    \[
    \dist_G(q, v_2) \geq \dist_G(C, v_2) > r_2 \geq r'_2 + \ell(K) = r'_2 + 2\cdot R_0(K) + 5K + 2,
    \]
    which by \ref{itm:CompAttToTwoBalls:Rad2} implies that $q \notin V^1_h$ for every $h \in B_{H^1}(h_2, 1)$. By \ref{itm:CompAttToTwoBalls:CompsNhood}, there is a bag $V^1_{h}$ of $(H^1, \cV^1)$ such that $q \in N_G(\hat{D}) \subseteq V^1_h$; by the previous observation, we have $h \notin B_{H^1}(h_2, 1)$, and hence $\rad_G(V^1_h) \leq R'_0(K)$ as we have shown in the proof of \ref{itm:CompAttToTwoBalls2:Rad1}.
    
    Now observe that
    \begin{equation} \label{eq:CompAttToTwoBalls2:1}
    N_G(D') = (N_G(D') \cap V(Y^1)) \cup (N_G(D') \cap B_2) \subseteq N_G(\hat{D}) \cup (N_G(D') \cap B_2)
    \end{equation}
    where we used that $D' \subseteq \hat{D} - B_2$.
    Since $D' \subseteq \hat D$, $N_G(\hat{D}) \subseteq V^1_h$ and $v_2 \notin V(\hat{D})$, every $D'$--$v_2$ path meets $V^1_h$.
    Thus, every vertex in $N_G(D') \cap B_2$ (which is empty if $h \notin \tilde H$) has distance at most $\ell(K)$ to $V^1_h$.
    Since also $N_G(\hat{D}) \subseteq V_h^1$, we have by \cref{eq:CompAttToTwoBalls2:1} for all $u, w \in N_G(D')$ that
    \[
    \dist_G(u, w) \leq \dist_G(u, V^1_h) + 2\cdot \rad_G(V^1_h) + \dist_G(V^1_h, w) \leq \ell(K) + 2 \cdot R'_0(K) + \ell(K) \leq R_2(K) - 1.
    \]
    This shows that $(H, \cV)$ satisfies \ref{itm:CompAttToTwoBalls2:CompsNhood} and thus concludes the proof.
\end{proof}

\begin{proof}[Proof of \cref{lem:CompWithThreeVertices}]
    Suppose that $G$ has no $K$-fat $K_4$ minor for $K \in \N_{\geq 1}$; otherwise, we are done.
    Since $4\cdot (2 (\ell(K) + 22K+1) + 11K + 2) = R_1(K)$, by \cref{lem:BallAndThreeComps} applied to $B_G(w, r + 22K+1)$ and $C^*$ (with $\ell := \ell(K) + 22K +1$) there exists a vertex $w' \in V(C^*)$ such that for every component $D$ of $G - (B_1 \cup B'_2)$ that attaches to $B_1 := B_G(w', 22K)$ and $B'_2 := B_G(w, r - \ell(K))$ there exists a $B_1$--$B'_2$ path~$P$ such that $\dist_G(G[D,1], P) \geq 5K$.
    Let $\cD$ be the set of components $D$ of $G - (B_1 \cup B'_2)$ which meet $C$ and attach to $B'_2$.
    We note that these also attach to $B_1$, as $w' \in V(C^*) \subseteq V(C)$ and $C$ is connected.
    Set $B_2 := B$.

    Note that $r > \ell(K)$, since there are two vertices in $\partial_G C^*$ which are at least $R_1(K) > 2(\ell(K)+22K + 1)$ apart.
    We may thus apply \cref{lem:CompAttToTwoBalls2} to $v_1 := w'$, $v_2 := w$, $r_1 := 22K$, $r_2 := r$ and every component $D \in \cD$. This yields for every such $D$ an honest partial \gd\ $\cH^{D} = (H^{D}, \cV^{D})$ of $G$ with support $Y^{D} \subseteq G[B_1 \cup V(D) \cup B_2]$ modelled on a graph $H^{D} \in \cH_{SP}$ with terminals $h_1^D, h_2^D$ which satisfies \cref{itm:CompAttToTwoBalls2:FirstBag} to \cref{itm:CompAttToTwoBalls2:CompsNhood}.

    We define the desired decomposition $(H, \cV)$ in multiple steps and check its desired properties afterwards. First, we define a decomposition $(H^1, \cV^1)$ of $Y^1 := G[\bigcup_{D \in \cD} V(Y^{D})] \subseteq G[B_1 \cup (\bigcup_{D \in \cD} V(D)) \cup B_2] \subseteq G[V(C) \cup B_2]$. The graph $H^1$ is obtained from the disjoint union of the $H^{D}$ and three new vertices $h_1, h_2,g$ by identifying all $h^{D}_1$ with $h_1$ as well as all $h^{D}_2$ with $h_2$ and adding the edge $h_2g$.
    Then we assign the nodes of $H^1$ bags $V_g^1 := N_G(C)$, $V^1_{h_2} := \bigcup_{D \in \cD} V^D_{h^D_2} \supseteq B_2$, and $V^1_{h_1} := \bigcup_{D \in \cD} V^D_{h^D_1} = B_1$ and $V^1_h := V^{D}_h$ for all other $h \in V(H^1)$ where $D$ is the unique component in $\cD$ with $h \in V(H^{D})$.
    Then $(H^1, \cV^1)$ is a  \gd\ of  $Y^1$.
    Indeed, it follows from $G[B_1 \cup B_2] \subseteq Y^D \subseteq G[B_1 \cup V(D) \cup B_2]$ for $D \in \cD$ that the $Y^D$ intersect only in $B_1$ and $B_2$.
    Hence, $(H^1, \cV^1)$ is a \gd\ of $Y^1$, as every $(H^D,\cV^D)$ is a \gd\ of $Y^D$ and $B_1 = V^D_{h^D_1}$ and $B_2 \subseteq V^D_{h^D_2}$.

    Next, we adjust $(H^1,\cV^1)$ to a  \gd\ $(H,\cV^2)$ of $Y^2 \supseteq Y^1$:
    Let $C'$ be a component of $G-Y^1$ which meets $C$.
    In particular, $C'$ is contained in $G-(B_1 \cup B_2)$ because $B_1, B_2 \subseteq V(Y^1)$.
    Let $D_{C'}$ be the unique component of $G- (B_1 \cup B'_2)$ which contains $C'$.
    By $\cC$ we denote the set of all components $C'$ of $G-Y^1$ which meet $C$ and whose $D_{C'}$ attach to $B'_2$.
    In particular, for every $C' \in \cC$ we have that $D_{C'} \in \cD$ and $C'$ is a component of $G-Y^{D_{C'}}$ that meets $C$ and $D_{C'}$, since $Y^{D_{C'}} \subseteq Y^1$.
    Hence, it follows from \cref{itm:CompAttToTwoBalls2:CompsNhood} of $(H^{D_{C'}}, \cV^{D_{C'}})$ that there is a node $h_{C'} \in V(H^{D_{C'}})$ such that $N_G(C') \subseteq V^{D_{C'}}_{h_{C'}}$ and $\rad_G(N_G(C')) \leq R_2(K)-1$.
    
    Now we obtain $H$ from $H^1$ by adding for each component $C' \in \cC$ a new node $h_{C'}'$ and the edge $h_{C'}h_{C'}'$.
    We assign an $h_{C'}'$ the bag $V^2_{h_{C'}'} := N_G(C') \cup \partial_G C'$ for $C' \in \cC$.
    Moreover, we set $V^2_h := V^1_h$ for all old nodes $h$ of $H^1 \subseteq H$.
    Then it is immediate from the construction that $(H,\cV^2)$ is a  \gd\ of  $Y^2 := G[V(Y^1) \cup (\bigcup_{C' \in \cC} \partial_G(C'))]$.
    We now obtain our final  \gd\ $(H, \cV)$ of  $Y := Y^2 \cap G[C,1]$ as the restriction of $(H,\cV^2)$ to $Y$.
    \smallskip
    
    It remains to check that $(H,\cV)$ has the desired properties.
    By definition, $Y \subseteq G[C,1]$.
    We also have $\partial_G C \subseteq V(Y)$. Indeed, every vertex $v \in \partial_G C \subseteq N_G(B_2)$ which was not already contained in $V(Y^1) \supseteq B_2$
    lies in $\partial_G C'$ for some component $C'$ of $G-Y^1$. 
    Then $C' \in \cC$, i.e.\ the component $D_{C'}$ of $G-(B_1 \cup B'_2)$ attaches to $B'_2$: because $v \in V(C') \subseteq V(D_{C'})$ and $B_2$ and $B_1$ are disjoint, every shortest $v$--$N_G(B'_2)$ path in~$G$ is through $B_2 \setminus B_2'$ and hence contained in $D_{C'}$. Now by construction, $v \in V_{h'_{C'}}$.
    Moreover, the graph $H$ is in $\Forbminor(K_4)$: The graph $H^1$ is obtained from the $H^D \in \cH_{SP}$ by parallel compositions of the graphs $H^{D} \in \cH_{SP}$ followed by a series composition with a path $h_2g$ with terminals $h_2$ and $g$; so $H^1$ is again in $\cH_{SP}$ with terminals $h_1$ and $g$ by \cref{obs:ConstrOfSPGraphs}. Now the graph $H$ is given by $1$-sums of the single edges $h_{C'}h_{C'}'$ with the graph $H^1 \in \cH_{SP} \subseteq \Forbminor(K_4)$. 
    Hence, $H \in \Forbminor(K_4)$.
    
    Since $V_g = V^2_g = V^1_g = N_G(C)$ by construction, $(H,\cV)$ satisfies \cref{itm:CompWithThreeVertices:FirstBag}.
    For \cref{itm:CompWithThreeVertices:Width}, we note that by \cref{itm:CompAttToTwoBalls2:Rad1} and \cref{itm:CompAttToTwoBalls2:RadiusWidth} of the $(H^D, \cV^D)$ we have $\orw((H^1, \cV^1)) \leq r + 2R'_0(K) + 1$. 
    For nodes $h \in H_2$ that are already in $H_1$, we have $V^1_h = V^2_h$, and hence $\rad_G(V^2_{h}) \leq \rad_G(V^1_{h})$.
    By construction of $(H^2,\cV^2)$, every node in $h \in H_2 - H_1$ is of the form $h'_{C'}$ for some $C' \in \cC'$, so their bags $V^2_{h'_{C'}}$ are contained in the respective $V^1_{h_{C'}} \cup N_G(V^1_{h_{C'}})$, and hence $\rad_G(V^2_{h'_{C'}}) \leq \rad_G(V^1_{h_{C'}}) + 1$.
    Thus, as $\rad_G(V_{h'_{C'}}) = \rad_G(V_{h'_{C'}}^2) \leq R_2(K)$, the \gd\ $(H, \cV)$ satisfies \cref{itm:CompWithThreeVertices:Width}.
    Also $(H, \cV)$ satisfies \cref{itm:CompWithThreeVertices:Spread}: The $(H^D, \cV^D)$ have radial spread at most $3$. Thus, $(H^1,\cV^1)$ has radial spread at most $6$ by construction.
    This yields that $(H, \cV^2)$ has radial spread at most $7$ which yields that its restriction $(H, \cV)$ also has radial spread at most $7$.

    We claim that $(H, \cV)$ satisfies \cref{itm:CompWithThreeVertices:CompsNhood}.
    Indeed, let $\tilde{C}$ be a component of $G-Y$.
    First, assume that $\tilde{C}$ does not meet $C$.
    Since $N_G(C) \subseteq V(Y) \subseteq B_G(C,1)$, its neighbourhood $N_G(\tilde{C})$ is contained in $N_G(C) = V_g$.
    Also, we have $\rad_G(N_G(C)) \leq \rad_G(B) \leq r$, since $C$ is a component of $G-B$ and $B = B_G(w,r)$. 
    So we may assume from now on that $\tilde{C}$ meets $C$.
    Since $N_G(C), B_1 \subseteq V(Y)$, the component $\tilde{C}$ is contained in $G- (B_1 \cup B_2)$, and thus in $C$.
    As $Y = Y^2 \cap G[C,1]$, $Y^2 \supseteq Y^1$ and $\tilde{C} \subseteq C$, there is a unique component $C'$ of $G - Y^1$ which contains $\tilde{C}$.
    If $D_{C'}$ does not attach to $B_2'$, its neighbourhood $N_G(D_{C'})$ is contained in $B_1 = V^1_{h_1} = V^2_{h_1} = V_{h_1}$, and thus $D_{C'} = C' = \tilde{C}$ by their respective definition.
    So we may assume that $D_{C'}$ attaches to $B_2'$, i.e.\ $C' \in \cC$.
    Then the construction of $(H, \cV^2)$ ensured that $V_{h_{C'}'} = V^2_{h_{C'}'} = N_G(C') \cup \partial_G C'$.
    Also $\tilde{C}$ is a component of $C' - \partial_G C'$, since $Y^2 = G[V(Y^1) \cup \bigcup_{C' \in \cC} \partial_G C']$.
    Thus, we have $N_G(\tilde{C}) \subseteq \partial_G C' \subseteq V^2_{h_{C'}'} = V_{h_{C'}'}$.
    We have already seen above that $\rad_G(V_{h_{C'}'}) \leq R_2(K)$.
\end{proof}

\section*{Acknowledgements}

The second named author gratefully acknowledges support by doctoral scholarships of the Studienstiftung des deutschen Volkes and the Cusanuswerk -- Bisch\"{o}fliche Studienf\"{o}rderung.
The third named author gratefully acknowledges support by a doctoral scholarship of the Studienstiftung des deutschen Volkes.

\bibliographystyle{amsplain}
\arXivOrNot{\bibliography{collectivearXiv.bib}}{\bibliography{collective.bib}}

\end{document}